\documentclass[a4paper]{amsart}
\usepackage[T1]{fontenc}
\usepackage[english]{babel}
\usepackage{amsmath,amsthm,color,multicol,enumerate}
\usepackage{amssymb}
\usepackage{xfrac}
\usepackage{framed}
\usepackage{microtype}
\usepackage{ulem}

\DeclareSymbolFont{forpolishl}{T1}{cmr}{m}{n}
\DeclareMathSymbol{\lucas}{0}{forpolishl}{'212}
\newcommand{\power}{\mathcal{P}}
\renewcommand{\emptyset}{\varnothing}
\newcommand{\lang}{\mathcal{L}}
\newcommand{\iimplies}{\rightarrow}
\newcommand{\Form}{\mathsf{Form}}
\newcommand{\Prop}{\mathsf{Prop}}
\newcommand{\logic}[1]{\mathsf{\mathbf{#1}}}
\newcommand{\framme}[1]{\mathfrak{#1}}

\newcommand{\model}[1]{\mathcal{#1}}
\newcommand{\Val}{\mathrm{Val}}
\newcommand{\free}{\mathfrak{F}}

\newcommand{\var}[1]{\mathcal{#1}}
\newcommand{\alg}[1]{\mathrm{\mathbf{#1}}}

\newcommand{\br}[1]{#1^\sharp}

\newcommand{\Oo}{\mathcal{O}}
\newcommand{\Cl}{\mathrm{Cl}}

\theoremstyle{plain}
\newtheorem{theorem}{Theorem}[section]
\newtheorem{lemma}[theorem]{Lemma}
\newtheorem{proposition}[theorem]{Proposition}

\theoremstyle{definition}
\newtheorem{definition}[theorem]{Definition}
\newtheorem{example}[theorem]{Example}

\theoremstyle{remark}

\newtheorem{remark}[theorem]{Remark}

\numberwithin{equation}{section}

\title[Modal Extensions of {\L}uk. Logic for Modeling Coalitional Power]{Modal Extensions of {\L}ukasiewicz Logic\\ for Modeling Coalitional Power}
\author{Tom\'a\v{s} Kroupa}
\address{Tom\'a\v{s} Kroupa\\ Institute of Information Theory and Automation of the Czech Academy of Sciences, Pod Vod\'arenskou v\v{e}\v{z}\'i~4, 182 08 Prague, Czech Republic}
\email{kroupa@utia.cas.cz}
\thanks{The work of Tom\'a\v{s} Kroupa was supported by grant GA \v{C}R P402/12/1309.}
\author{Bruno Teheux}
\address{Bruno Teheux\\ University of Luxembourg, Facult\'e des Sciences, de la Technologie et de la Communication, 6, rue Richard Coudenhove-Kalergi 
L-1359, Luxembourg}
\email{bruno.teheux@uni.lu}
\thanks{The work of Bruno Teheux was supported by the internal research project F1R-MTH-PUL-15MRO3 of the University of Luxembourg.}

\begin{document}
\normalem
\begin{abstract}
Modal logics for reasoning about the power of coalitions capture the notion of effectivity functions associated with game forms. The main goal of coalition logics is to provide formal tools for modeling the dynamics of~a~game frame whose states may correspond to different game forms. The~two classes of effectivity functions studied are the families of playable and truly playable effectivity functions, respectively. In this paper we generalize the~concept of effectivity function beyond the yes/no truth scale. This enables us to describe the 
situations in which the coalitions assess their effectivity in degrees, based on functions over the outcomes taking values in a finite \L ukasiewicz chain. Then we introduce two modal extensions of \L ukasiewicz finite-valued logic together with many-valued neighborhood semantics in order to encode the properties of many-valued effectivity functions associated with game forms. As our main results we prove completeness theorems for the two newly introduced modal logics. 
\end{abstract}
\maketitle

\section{Introduction}
Modeling collective actions of agents and capturing their effectivity is among the important research topics on the frontiers of game theory, computer science and mathematical logic. The main efforts are concentrated on answering the following question: what is the set of outcome states that can effectively be implemented by a~coalition of agents? A game-theoretic framework for studying collective actions and their enforceability is based on the notion of game forms. Loosely speaking, a~game form is a pure description of a game and its rules, without regard to the agents' preferences. Game frames enable us to capture a more general action model in which a game form is associated with every state of the frame and the outcome states of the game forms are the states of the frame. From the game-theoretic viewpoint, the game frames are extensive form games with simultaneous moves of~the players; see \cite{OsborneRubinstein94,vanderHoekPauly07}.

The concept of $\alpha$-effectivity  (\cite{AbdouKeiding91,Peleg98}) is one of the key approaches to characterize the coalitional effectivity within game form models. A coalition $C$ is $\alpha$-effective for a~set of outcome states $X$ if the players in $C$ can choose a joint strategy that enforces an outcome in $X$ no matter what strategies are adopted by the other players. The previous definition gives rise to the concept of a (truly) playable effectivity function. In his seminal paper \cite{Pauly2002}, Pauly introduces  Coalition Logic $\logic{CL}_N$ to reason about $\alpha$-effectivity in  game forms with player set $N$. The axiomatization of~$\logic{CL}_N$ is an attempt to characterize  the class of $\alpha$-effectivity functions in a~multi-modal language. Pauly also defined a~neighborhood semantics with respect to which $\logic{CL}_N$ is complete. The logic $\logic{CL}_N$ was subsequently analyzed and extended by many authors; see \cite{AgotnesHoekWooldridge09,BoellaGabbayGenoveseTorre10}. In particular, Goranko et al. \cite{Goranko2013} found a gap in Pauly's characterization of playable effectivity functions (see \cite[Theorem 3.2]{Pauly2002}), which led to the introduction of truly playable effectivity functions.\footnote{We are grateful to Paolo Turrini who brought  the paper \cite{Goranko2013} to our attention.} The reader is invited to consult Section \ref{sect:eff} together with Appendix \ref{appendix:TP}, where we recall all the necessary notions regarding Boolean effectivity functions and the distinction between playable and truly playable functions, respectively.

In this paper we extend the results of Pauly \cite{Pauly2002} and Goranko et al. \cite{Goranko2013} to the situations in which the effectivity of coalitions is evaluated on a finer finite scale than $\{0,1\}$. This generalization is based on several assumptions. First, we assume that coalitions evaluate their effectivity with respect to a certain family of $[0,1]$-valued functions over the state space and not only with respect to the sets of states. Second, the effectivity of coalitions comes in degrees rather than in Boolean values. Third, the underlying logical framework is that of finitely-valued {\L}ukasiewicz logic, which offers a great expressive power, while preserving many desirable properties of logics at the same time.
 
Our main goal is to investigate the properties of many-valued modal logics $\logic{P}_n$ and $\logic{TP}_n$ devised for capturing the refined notion of effectivity. To this end, we proceed as follows. In Section \ref{sec:MVefffunctions} we generalize the notion of $\alpha$-effectivity. Since we use finitely-valued {\L}ukasiewicz logic, our scale is always the set
\begin{equation}\label{def:chain}
\lucas_n=\left\{0, \tfrac{1}{n}, \ldots, \tfrac{n-1}{n}, 1\right\},\enskip \text{where $n$ is a positive integer.}\footnote{Our notation for $\lucas_n$ as a set of cardinality $n+1$ follows \cite{Mundici11}.}
\end{equation}
\noindent
The choice of {\L}ukasiewicz logic is not only a design choice that is suitable for applying the usual operations of MV-algebras to the $\lucas_n$-valued functions over the state space, but also a matter of practical necessity since the  homogeneity property (Definition \ref{defn:nbg}) is among the key features of (playable) many-valued effectivity functions. On the logical side it corresponds to the axioms (1)--(2) of the logic $\logic{P}_n$ in Definition \ref{defn:log}. It is worth mentioning that homogeneity appears as an axiom in other modal languages as well; see \cite{Bou2011,Ost88,Teheux2014}.  The finiteness of our scale of truth degrees is essential for the completeness results, which usually fail badly for modal extensions of infinitely-valued {\L}ukasiewicz logic \cite{Teheux2012}.

We introduce the concept of $\lucas_n$-valued effectivity function whose purpose is to capture the effectivity of coalitions on the scale $\lucas_n$. In order to understand the relation between effectivity functions and game forms, we have to consider the class of playable and truly playable $\lucas_n$-valued effectivity functions, respectively, which are studied in Section \ref{sec:MVef}.
 In particular, we establish a characterization of truly playable $\lucas_{n}$-valued effectivity functions (Theorem \ref{thm:TP}). In Section \ref{sec:coal}, we develop tools to capture the properties of $\lucas_{n}$-valued effectivity functions in a many-valued modal language. These developments rely not only on recent advances in modal extensions of \L ukasiewicz logic (see \cite{Bou2011, Teheux2012}), but they also require the introduction of neighborhood semantics, which has never been considered in the modal many-valued setting before, to the best of our knowledge. The newly introduced logics $\logic{P}_n$ and $\logic{TP}_n$ axiomatize in the many-valued modal language the properties of playable and truly playable $\lucas_{n}$-valued effectivity functions, respectively. Our main results are Theorem~\ref{thm:comp} and Theorem \ref{thm:comp02}, which show that the logics $\logic{P}_n$ and $\logic{TP}_n$ are complete with respect to the corresponding classes of $\lucas_{n}$-valued coalitional frames. The key ingredient in the proof of completeness of $\logic{TP}_{n}$ is the many-valued generalization of the filtration technique for neighborhood models \cite[Chapter 7.5]{Chellas1980}. These mathematical constructions have their own merit and are among the main contributions of the paper, since they provide a ``bag of tricks'' that could be reused to develop neighborhood semantics for other  modal extensions of finite-valued \L ukasiewicz logics.

\section{Game forms and effectivity functions}\label{sect:eff}
We recall basic facts about game forms and effectivity functions; see \cite{AbdouKeiding91}.
In what follows, $S$ denotes a nonempty set,  $N=\{1, \ldots, k\}$ is a finite set and, for any $i \in N$, $\Sigma_i$ denotes a nonempty set. For any set $\Omega$ we denote by $\power \Omega$ the powerset of $\Omega$.

\begin{definition}[{\cite[Definition 3.1]{AbdouKeiding91}}]\label{def:GF}
A \emph{(strategic) game form} is a tuple $G=(N,\{\Sigma_i \mid i \in N\}, S, o)$, where $N$ is a~set of \emph{players}, $\Sigma_{i}$ is a set of \emph{strategies} for each $i\in N$, $S$ is a set of \emph{outcome states}, and $o\colon \prod_{i\in N}\Sigma_i \to S$ is an \emph{outcome function}.
\end{definition}

\noindent The game forms are not to be confused with strategic games. While a preference relation over $S$ must be defined for each player $i\in N$ in a strategic game \cite{OsborneRubinstein94}, no~such requirement exists for a game form. Below we provide some basic examples of~game forms.

\begin{example}
	
(i) Let $N=\{1,2\}$ and $\Sigma_{1}$, $\Sigma_{2}$ be some strategy sets. Assume that the players choose their strategies simultaneously. Then we may set  $S=\Sigma_{1}\times \Sigma_{2}$ and define $o$ as the identity function, which turns  $(\{1,2\},\{\Sigma_{1},\Sigma_{2}\},\Sigma_{1}\times\Sigma_{2},o)$ into a~game form. This game form just records an outcome as the pair of chosen strategies.
(ii) Suppose, on the other hand, that Player 2 makes his choice only after observing the strategic choice of Player 1. This sequential procedure is modeled by a game form such that $\Sigma_{2}$ is the set of all functions $r\colon\Sigma_{1}\to \Sigma_{2}'$, where $\Sigma_{2}'$ can be viewed as the set of all possible  moves that can be played by Player $2$. Hence, $\Sigma_2$ models the replies of Player 2 to the selection of a strategy by Player~1. The~outcome function is given by $o(\sigma_{1},r)=(\sigma_{1},r(\sigma_{1}))$, where $(\sigma_{1},r)\in \Sigma_{1}\times \Sigma_{2}$ and the set of outcome states is $S=\Sigma_{1}\times \Sigma_{2}'$.

\end{example}

An important example arises when the outcome function coincides with some social choice correspondence in the sense of \cite[Chapter 1]{AbdouKeiding91}.
\begin{example}\label{ex:SCorr}
Let $S'$ be any nonempty set of outcome states and $\Pi(S')$ be a~set of~admissible preference relations on $S'$. In most applications, $\Pi(S')$ will be either the set of \emph{total preorders} (reflexive, transitive, and complete binary relations) or the set of \emph{linear orders}. A map $\pi\colon\Pi(S')^N \to \power S'$ is called a \emph{social choice correspondence}. Social choice correspondences implement collective decision procedures mapping a~preference profile $\sigma\in \Pi(S')^N$ of the agents into a set of outcome states that are considered equivalent with respect to $\sigma$. If an agent (or a group of agents) wants to enforce a specific outcome, his/her only possible strategy is to declare a~preference relation that is likely to bring the collective decision into an outcome $\pi(\sigma)$ that contains the desired state. We can describe this scheme as a game form $G=(N, \{\Sigma_i \mid i \in N\}, S, o)$ in which $\Sigma_1=\dotsb = \Sigma_{k}=\Pi(S')$, $S=\power S'$ and $o=\pi$.  
\end{example}

The subsets $C\subseteq N$ are called \emph{coalitions}.  For every coalition $C$ we denote by $\overline{C}$ its set-complement in $N$. If $\sigma_C \in \prod_{i \in C} \Sigma_i$ and $\sigma_{\overline{C}} \in \prod_{i \in \overline{C}} \Sigma_i$, then $\sigma_C\sigma_{\overline{C}}$ is the strategy  tuple in $\prod_{i \in N} \Sigma_i$ defined by $(\sigma_C\sigma_{\overline{C}})_i=(\sigma_C)_i$ if $i \in C$ and  $(\sigma_C\sigma_{\overline{C}})_i=(\sigma_{\overline{C}})_i$ if $i \in \overline{C}$.

\begin{definition}[{\cite[Definition 4.1]{AbdouKeiding91}}]\label{def:GFeffect}
	Let $G$ be a game form. The \emph{effectivity function of $G$} is the mapping $H_G\colon \power N \to \power\power S$ defined as follows:
$X \in H_{G}(C)$ if there exists $\sigma_{C}$ such that for every $\sigma_{\overline{C}}$, we have $o(\sigma_{C}\sigma_{\overline{C}})\in X$.
\end{definition}
\noindent
In other words, the condition $X\in H_{G}(C)$ is true whenever the coalition $C$ has the~power to force the outcome to lie in $X$. We refer to \cite{AbdouKeiding91, Pauly2002,Peleg98} for a~discussion and examples of effectivity functions in game theory and social choice. 

The notion of effectivity function of a game form can be generalized as follows.
\begin{definition}
A mapping $E\colon \power N  \to \power\power S$ is called an \emph{effectivity function}.
\end{definition}
\noindent
The problem of characterizing effectivity functions that are effectivity functions of game forms is solved by introducing the notion of true playability.  It was used by Goranko et al. \cite{Goranko2013} in order to fix the error in Pauly's characterization result  \cite[Theorem 3.2]{Pauly2002}, which was based on the weaker notion of playability. We recall the two playability concepts in the next definition.
\begin{definition}\label{def:BooleanPlay}
Let $E\colon \power N \to \power\power S$ be an effectivity function. We say that $E$
\begin{enumerate}
\item is \emph{superadditive}\label{it:booleSA} if $C_{1}\cap C_{2}=\emptyset$ and $X\in E(C_{1})$, $Y\in E(C_{2})$ imply that $X\cap Y \in E(C_{1}\cup C_{2})$, for every $C_{1},C_{2}\in \power N$;
\item is \emph{outcome monotonic} if $X\in E(C)$ and $X\subseteq Y$ imply $Y\in E(C)$, for every $C\in \power N$;
\item is \emph{$N$-maximal} if $\overline{X}\notin E(\emptyset)$ implies $X\in E(N)$;
\item has the \emph{liveness property} if $\emptyset \notin E(C)$, for every $C\in \power N$;
\item has the \emph{safety property} \label{it:boolesafe} if $S \in E(C)$, for every $C\in \power N$.
\end{enumerate} 
We call $E$ \emph{playable} whenever \eqref{it:booleSA}--\eqref{it:boolesafe} are satisfied. We say that $E$ is \emph{truly playable} if it is playable and $E(\emptyset)$ is a principal filter in $\power S$.\footnote{Our formulation of true playability is different from the original definition yet equivalent to it by \cite[Proposition 5]{Goranko2013}.}
\end{definition}

The following result, which was originally proved in \cite[Theorem 1]{Goranko2013}, amends the gap in the proof of Pauly's correspondence result in case of an infinite outcome space $S$. We provide an alternative proof of Theorem \ref{thm:Goranko} in Appendix~\ref{appendix:TP}, which shows that this result can be considered as a corollary  of Peleg's Theorem {\cite[Theorem 3.5*]{Peleg98}}.

\begin{theorem}[{\cite[Theorem 1]{Goranko2013}}]\label{thm:Goranko}
Let $E\colon \power N \to \power\power S$ be an effectivity function. There exists a game form $G=(N, \{\Sigma_i \mid i \in N\}, S, o)$ satisfying $E=H_{G}$ if and only if $E$ is truly playable.
\end{theorem}

\section{Many-valued effectivity functions}\label{sec:MVefffunctions}
We are going to generalize the concept of effectivity function for an arbitrary game form $G=(N,\{\Sigma_i \mid i \in N\}, S, o)$ and \L ukasiewicz chain $\lucas_n$ as in (\ref{def:chain}). Our goal is to capture the degree or extent to which a coalition $C$ can ``enforce'' a  function $f\colon S \to \lucas_n$. Before stating a formal definition, we will motivate this idea by two situations where such many-valued assessments $f$ may arise.
\begin{enumerate}
	\item\label{item:util} A strategic game form $G$ is made into a strategic game when a utility function (or a preference relation) over the outcome set $S$ is introduced for every player $i\in N$. Thus an arbitrary function $f \in \lucas_n^S$ can be viewed as a~utility function. However, this utility function is not necessarily attached to any player's preference relation. 
	\item\label{item:topo} When the state space $S$ is too large or complex to deal with, the distinction between the subsets of $S$ (equivalently, the functions $S\to \{0,1\}$) and the functions $f \in \lucas_n^S$ may become immaterial. It is not against the spirit of neighborhood semantics to draw a direct parallel with an analogous situation in topology: by Urysohn's lemma any two closed disjoint subsets in a~normal topological space can be arbitrarily closely approximated by a~$[0,1]$-valued continuous function. Thus, for a sufficiently large natural number $n$, we may think of functions $f \in \lucas_n^S$ as members of some limit sequence, which eventually encodes a subset of $S$. This interpretation of an originally finite object is not uncommon in game theory. Indeed, it was one of the motivations for the development of Aumann's theory of ideal coalitions in coalition games with continuum of players; see \cite{AumannShapley74}.
\end{enumerate}

\noindent

\begin{definition}\label{def:GFeff}
Let $G=(N,\{\Sigma_i \mid i \in N\}, S, o)$ be a game form. The $\lucas_n$\emph{-valued effectivity function of $G$} is the map $E_G\colon\power N \times \lucas_n^S \to \lucas_n$ defined by
 \begin{equation}\label{def:eff}
E_{G}(C,f)=\max_{\sigma_C}\min_{\sigma_{\overline{C}}} f(o(\sigma_C\sigma_{\overline{C}})), \qquad C \in \power N, f \in \lucas_n^S,
\end{equation}
where $\sigma_C$ and $\sigma_{\overline{C}}$ range through the set of all joint strategies of coalitions $C$ and~$\overline{C}$, respectively.
\end{definition}

The meaning of definition \eqref{def:eff} is the following: coalition $C$ is effective for $f \in \lucas_n^S$ to the degree $E_{G}(C,f)\in \lucas_n$, disregarding the strategic options of players in the  opposite coalition $\overline{C}$.  Note that the usual Boolean $\alpha$-effectivity function associated with $G$ coincides with the $\lucas_1$-valued effectivity function of $G$.

\begin{remark}\label{rem:adv}
In this paper we do not advocate any particular interpretation of the many-valued effectivity model \eqref{def:eff} as suggested by (\ref{item:util})--(\ref{item:topo}) above, nor do we insist on a special meaning of truth degrees. From the purely mathematical point of view, any such interpretation is irrelevant since it yields the same underlying game form under the assumption of several playability conditions introduced in Section \ref{sec:MVef}. This point will be explained in detail in Remark \ref{rem:MVef}.
\end{remark}

\section{Playability of $\lucas_n$-valued effectivity functions}\label{sec:MVef}
Analogously to the classical literature \cite{Moulin83,Peleg98} on effectivity functions, we can study the notion of effectivity in a setting independent of game forms. Let $S$ be a~set of outcomes and $N$ be a finite player set. We always assume that $|S|\geq 2$ and $|N|\geq 2$.
\begin{definition}An \emph{$\lucas_n$-valued effectivity function} is a mapping $E\colon \power N \times \lucas_n^{S} \to \lucas_n$.
\end{definition} 
Note that the $\lucas_1$-valued effectivity functions are exactly the effectivity functions $\power N\times \{0,1\}^{S}\to \{0,1\}$ arising in the Boolean framework \cite{AbdouKeiding91,Pauly2002}. Therefore we call any $\lucas_{1}$-valued effectivity function a \emph{Boolean effectivity function}. 

Our goal is to characterize the class of $\lucas_n$-valued effectivity functions that are associated with game forms. This characterization is related to the properties of effectivity functions listed in Definition \ref{defn:nbg}. We use the standard connectives of \L ukasiewicz logic and respective operations of MV-algebras; see Appendix \ref{appendix:MV}. In~particular, we always apply the operations of the  MV-algebra $\lucas_n$ to functions $f$ in $\lucas_n^S$ pointwise.  For any $\lucas_n$-valued effectivity function $E$, we define
\(E(\varnothing,-)^{-1}(1)=\{f \in \lucas_n^S \mid E(\varnothing,f)=1\}\).

\begin{definition}\label{defn:nbg} Let $E$ be an $\lucas_n$-valued effectivity function. We say that $E$  
\begin{enumerate}
\item\label{it:opi01} is \emph{outcome monotonic} whenever $f\geq g$ implies $E(C, f)\geq E(C, g)$, for every $C \in \power N$ and every $f,g\in \lucas_n^S$;
\item\label{it:opi03} is \emph{$N$-maximal} if  $\neg E(\varnothing, \neg f)\leq E(N, f)$ for every $f\in \lucas_n^S$;
\item\label{it:opi05} is \emph{regular} if it satisfies $E(C,f)\leq \neg E(\overline{C}, \neg f)$ for every $C\in \power N$ and $f\in \lucas_n^S$;
\item\label{it:opi04} is \emph{superadditive} if  $E(C_1, f)\wedge E(C_2, g)\leq E(C_1\cup C_2, f\wedge g)$ for every pair of~coalitions $C_1, C_2 \in \power N$ such that $C_1 \cap C_2=\varnothing$ and every $f,g \in \lucas_n^S$;
\item\label{it:opi06} is \emph{coalition monotonic} if $E(C,f)\leq E(C',f)$ for every $C\subseteq C' \in \power N$ and every $f \in \lucas_n^S$;
\item\label{it:opi012} is \emph{homogeneous} if $E(C, f\oplus f)=E(C,f) \oplus E(C,f)$ and $E(C, f\odot f)=E(C,f) \odot E(C,f)$ for every $C \in \power N$ and every $f\in \lucas_n^S$;
\item\label{it:opi012bis} has the \emph{liveness property} if $E(C,1)=1$ for every $C\in \power N$;
\item\label{it:opi012ter} has the \emph{safety property} if $E(C,0)=0$ for every $C\in \power N$;
 \item\label{it:opi013} is \emph{principal} if there exists some  $g \in \lucas_n^S$ such that   $\{f \in \lucas_n^S \mid E(\varnothing, f)=1\}=\{f \in \lucas_n^S \mid f\geq \bigodot_{i=1}^{n}g\}$.
\end{enumerate}
We say that $E$ is \emph{playable} whenever it is outcome monotonic, $N$-maximal, superadditive, homogeneous, and has  liveness and safety properties. We say that $E$ is \emph{truly playable} if it is playable and principal. 
\end{definition}
If $n=1$, then the definitions of (truly) playable Boolean effectivity function coincide with the corresponding definitions used in the Boolean setting \cite{Pauly2002,Goranko2013}; cf.~Definition~\ref{def:BooleanPlay}.

Note that if $E$ is an outcome monotonic and homogeneous $\lucas_{n}$-valued effectivity function, then $E(C,-)^{-1}(1)$ is an MV-filter of the MV-algebra $\lucas_n^S$ \cite{Cignoli2000}. Moreover, the~$\lucas_{n}$-valued effectivity function $E$ is principal whenever the MV-filter $E(\varnothing,-)^{-1}(1)$ is principal; see Appendix \ref{appendix:MV}.

It may be difficult to get some intuition about the definition of a homogeneous $\lucas_n$-valued effectivity function in the game form framework. We refer to Remark~\ref{rem:equiv} for an equivalent formulation of this definition. The following result illustrates that homogeneity arises naturally in the context of game forms. For every set $Y\subseteq S$, we denote by $\chi_Y$ the characteristic function of $Y$.

\begin{proposition}\label{prop:nxq}
If $G=(N, \{\Sigma_i \mid i \in N\}, S, o)$ is a game form, then $E_G$ is a~truly playable $\lucas_n$-valued effectivity function.
\end{proposition}
\begin{proof}
It follows directly from Definition \ref{def:GFeff} that $E_G$ is outcome monotonic, $N$-maximal and has liveness and safety. Homogeneity of $E_G$  follows from the fact that the maps $\tau_\oplus\colon x \mapsto x \oplus x$ and $\tau_\odot\colon x \mapsto x \odot x$ are lattice homomorphisms of $\lucas_n$. Moreover, $E_G(\varnothing, -)^{-1}(1)=\{g \mid g \geq \chi_{\mathrm{ran}(o)}\}$, which shows that $E_G$ is principal since $\bigodot_{i=1}^{n}\chi_{\mathrm{ran}(o)}=\chi_{\mathrm{ran}(o)}$.

It remains to prove that $E_G$ is superadditive. Let $C_1, C_2\in \power N$ be such that $C_1 \cap C_2=\varnothing$ and $f_1, f_2\in \lucas_n^S$. Denote by $\sigma_{C_1}^*$ and $\sigma_{C_2}^*$ two strategy tuples satisfying
$E(C_1, f_1)=\min_{\sigma_{\overline{C}_1}} f_1(o(\sigma_{C_1}^*\sigma_{\overline{C}_1}))$  and 
$E(C_2, f_2)=\min_{\sigma_{\overline{C}_2}} f_2(o(\sigma_{C_2}^*\sigma_{\overline{C}_2}))$.
It follows  that $E(C_1, f_1) \wedge E(C_2, f_2) =  \min_{\sigma_{\overline{C}_1}} \min_{\sigma_{\overline{C}_2}} f_1(o(\sigma_{C_1}^*\sigma_{\overline{C}_1})) \wedge f_2(o(\sigma_{C_2}^*\sigma_{\overline{C}_2}))$, which is not greater than $\min_{\sigma_{\overline{C}_1 \cap \overline{C}_2} }(f_1\wedge f_2)(o(\sigma_{C_1}^*\sigma_{C_2}^*\sigma_{\overline{C}_1 \cap \overline{C}_2}))$. The conclusion follows from the fact that the latter  is bounded above by $E_G(C_1 \cup C_2, f_1\wedge f_2)$.
\end{proof}

As the next result shows, a Boolean effectivity function can be associated with any homogeneous $\lucas_n$-valued effectivity function. The Boolean algebra $\lucas_1^S=\{0,1\}^{S}$ is called the \emph{Boolean skeleton of $\lucas_n^S$}. In other words, the Boolean skeleton of $\lucas_n^S$ is the powerset of $S$ if we identify the subsets of $S$  with their characteristic functions on $S$. An element $f\in \lucas_n^S$ belongs to the Boolean skeleton of $\lucas_n^S$ if and only if $f\oplus f=f$, and such an element is said to be \emph{idempotent}. For every $\lucas_n$-valued effectivity function $E\colon \power N \times \lucas_n^S \to \lucas_n$, we denote by $\br{E}$ the restriction of $E$ to $\power N \times \lucas_1^S$.

\begin{lemma}\label{lem:uip}
If $E$  is a homogeneous  $\lucas_n$-valued effectivity function, then 
$\br{E}$ is a~Boolean effectivity function. If in addition $E$ is playable (respectively, truly playable), then so is $\br{E}$.
\end{lemma}
\begin{proof}
For any idempotent element $f\in \lucas_n^S$ and for every $C\in \power N$, we obtain
\[
E(C,f)\oplus E(C,f)=E(C,f\oplus f)=E(C,f).
\]
Therefore $\br{E}$ is a Boolean effectivity function.

If in addition $E$ is playable (respectively, truly playable), then it satisfies conditions (\ref{it:opi01}), (\ref{it:opi03}), (\ref{it:opi04}), (\ref{it:opi012bis}) and (\ref{it:opi012ter}) (respectively, conditions (\ref{it:opi01}), (\ref{it:opi03}), (\ref{it:opi04}), (\ref{it:opi012bis}),  (\ref{it:opi012ter}) and  (\ref{it:opi013}))  of Definition \ref{defn:nbg}. It follows that $\br{E}$ also satisfies the analogous Boolean conditions (see Appendix \ref{appendix:TP}) since they do not involve any existential quantifier over the elements of $\lucas_n^S$.
\end{proof}
\noindent
In order to study the playability property, we need more technical preliminaries. To this end, put
\[
\tau_\oplus(x)=x\oplus x \enskip \text{and} \enskip \tau_\odot(x)=x\odot x, \enskip \text{for every $x\in\lucas_{n}$.}
\]

\begin{definition}\label{defi:termes} 
Let $i\in\{1, \ldots, n\}$. We define the function $\tau_{\sfrac{i}{n}}\colon \lucas_{n}\to \lucas_{n}$ by
\[
\tau_{\sfrac{i}{n}}(x)=\begin{cases}
0 & x<\frac{i}{n},\\
1 & x\geq \frac{i}{n},\end{cases}
\]
and we always assume that $\tau_{\sfrac{i}{n}}$ is the interpretation on $\lucas_n$ of an algebraic term which is a composition of finitely many copies of the maps $\tau_\oplus$ and $\tau_\odot$  alone.\footnote{The proof of existence of such a term appears in \cite{Ost88}. }
\end{definition}

Any mapping  $\tau\colon \lucas_{n}\to\lucas_{n}$ can be composed with a function $f\in \lucas_{n}^{S}$. Thus we define  $\tau(f)(s)=\tau(f(s))$ for every $s\in S$ and $f\in \lucas_{n}^{S}$.
\begin{lemma}\label{lem:vaz}
Let $E, E'\colon\power N \times \lucas_n^S \to \lucas_n$ be homogeneous $\lucas_n$-valued effectivity functions. Then $E=E'$ if and only if $\br{E}=\br{E'}$.
\end{lemma}
\begin{proof}
Necessity is trivial. To prove sufficiency assume that $\br{E}=\br{E'}$. Let $C\in \power N$, $f\in \lucas_n^S$ and $i \in \{1, \ldots ,n\}$. Since $\tau_{\sfrac{i}{n}}(f)$ is idempotent and $E$ and $E'$ are homogeneous, we have $E(C,f)\geq \frac{i}{n}$ if and only if $1= E(C, \tau_{\sfrac{i}{n}}(f))= E'(C, \tau_{\sfrac{i}{n}}(f))$, which is equivalent to $E'(C, f)\geq \frac{i}{n}$.
\end{proof}
\noindent
The following lemma is straightforward. Its statement uses the notion of the Boolean effectivity function $H_{G}$ associated with a game form $G$; see Definition~\ref{def:GFeffect}.
\begin{lemma}\label{lem:tri}
Let $G$ be a game form. If $H_G$ and $E_G$ are the Boolean and the $\lucas_n$-valued effectivity function associated with $G$, respectively,  then $H_G=E_G^\sharp$.
\end{lemma}

For any $r\in [0,1]$, we denote by $\lceil r\rceil$ the element $\min\{a \in \lucas_n \mid a \geq r\}$. The following lemma will turn out to be crucial for understanding the limits of expressive power of the language associated with (truly) playable $\lucas_n$-valued effectivity functions; see Proposition \ref{prop:imp}.
\begin{lemma}\label{lem:new}
If $H\colon \power N \times \lucas_1^S  \to \lucas_1$ is a playable Boolean effectivity function, then the function $E\colon\power N \times \lucas_n^S \to \lucas_n$ defined by 
\begin{equation}\label{eq:defBEF}
E(C, f)=\max\left\{\tfrac in\in\lucas_{n} \mid H(C,\tau_{\sfrac in}(f))=1\right\},
\end{equation}
for every $C\in\power N$ and $f\in \lucas_{n}^{S}$, is a playable $\lucas_n$-valued effectivity function that satisfies $E^\sharp=H$. If in addition $H$ is truly playable, then so is $E$.
\end{lemma}
\begin{proof}
Clearly $E^\sharp=H$. It follows from outcome monotonicity of $H$ that $E$ is outcome monotonic. Since $H$ has  liveness and safety, so does the function $E$. To prove that $E$ is superadditive, assume on the contrary that there exist $f,g \in \lucas_n^S$ and $C, D \in \power N$ such that $C \cap D=\varnothing$ and $E(C\cup D, f\wedge g)< \frac{i}{n}\leq E(C,f)\wedge E(D,g)$ for some $i \in\{1, \ldots, n\}$. On the one hand, it follows that $H(C\cup D, \tau_{\sfrac{i}{n}}(f) \wedge \tau_{\sfrac{i}{n}}(g))=H(C\cup D, \tau_{\sfrac{i}{n}}(f \wedge g))=0$. On the other hand, we obtain $H(C, \tau_{\sfrac{i}{n}}(f))=H(D, \tau_{\sfrac{i}{n}}(g))=1$ and by superadditivity of $H$ we get $H(C\cup D, \tau_{\sfrac{i}{n}}(f) \wedge \tau_{\sfrac{i}{n}}(g))=1$, a contradiction.

Now, let $f \in \lucas_n^S$. Since $\tau_{\sfrac{i}{n}}(f\oplus f)=\tau_{\lceil\sfrac{i}{2n}\rceil}(f)\in \lucas_1^S$, it follows from the definition of $E$ that $E(C, f \oplus f) \geq \frac{i}{n}$ if and only if  $H(C, \tau_{\lceil\sfrac{i}{2n}\rceil}(f))=1$. Using again the definition of $E$, the latter is equivalent to $E(C, f)\geq \lceil\frac{i}{2n}\rceil$, which is the same as $ E(C, f) \oplus E(C, f) \geq \frac{i}{n}$.
We can proceed in a similar way to prove that $E(C, f\odot f)=E(C,f)\odot E(C,f)$ for every $f\in \lucas_n^S$.

To prove $N$-maximality of $E$, consider $f\in \lucas_n^S$ and $i\in \{0, \ldots, n-1\}$. It follows from the definition of $E$ and $N$-maximality of $H$ that $E(N,f) \leq \frac{i}{n}$ if and only if  $H(\varnothing, \neg \tau_{\sfrac{(i+1)}{n}}(f))=1$. Since $ \neg \tau_{\sfrac{(i+1)}{n}}(f)= \tau_{\sfrac{(n-i)}{n}}(\neg f)$, the last identity is equivalent to $H(\varnothing, \tau_{\sfrac{(n-i)}{n}}(\neg f))=1$ and finally to $ \neg E(\varnothing, \neg f) \leq \frac{i}{n}$.

Assume that $H$ is truly playable. Definition \ref{def:BooleanPlay} yields existence of $g\in \lucas_1^S$ such that $H(\varnothing, -)^{-1}(1)=\{h \in \lucas_1^S \mid h \geq g\}$. Since $E$ is homogeneous, $E(\varnothing, -)^{-1}(1)=\{f \in \lucas_n^S \mid H(\varnothing, \tau_1(f))=1\}=\{f \in \lucas_n^S \mid \tau_1(f) \geq g\}$. The latter is equal to $\{f\in \lucas_n^S \mid f \geq g\}$ since $E(\varnothing, -)^{-1}(1)$ is an MV-filter of $\lucas_n^S$ and $g$ is idempotent.
\end{proof}

The following result, which is the $\lucas_n$-valued generalization of \cite[Theorem 3.2]{Pauly2002} and \cite[Theorem 1]{Goranko2013}, completes the characterization of truly playable $\lucas_n$-valued effectivity functions. 
\begin{theorem}\label{thm:TP}
An $\lucas_n$-valued effectivity function $E\colon \power N \times \lucas_n^S \to \lucas_n$ is truly playable if and only if there is a game form $G$ such that $E=E_G$.
\end{theorem}
\begin{proof}
By Proposition \ref{prop:nxq}, $E_{G}$ is truly playable. Conversely, assume that $E$ is truly playable. By Lemma \ref{lem:uip} and Theorem \ref{thm:Goranko}, there exists a game form $G$ such that $H_G=\br{E}$, where $H_G$ is the Boolean effectivity function of $G$. We obtain by Lemma~\ref{lem:tri} that $\br{E}=H_G=\br{E_G}$, where $E_G$ is the $\lucas_n$-valued effectivity function associated with $G$. The conclusion $E=E_G$ follows from Lemma \ref{lem:vaz}.
\end{proof}

\begin{remark}\label{rem:MVef}
The previous theorem implies that the notions of playability for Boolean effectivity functions and $\lucas_n$-valued functions are equivalent on the game-theoretic level since any of those concepts leads to a uniquely determined game form. In our more general setting we were able to maintain the correspondence with game forms by imposing homogeneity and the $\lucas_n$-version of the playability axioms in the case of the $\lucas_n$-valued effectivity functions. Admittedly, we do not arrive at a new concept of game form or extend the validity of Boolean effectivity functions to a larger class of objects. Nevertheless, the importance of our approach presented herein lies in an alternative representation of the classical setting: allowing for a richer language and more truth degrees leaves the underlying class of game forms invariant. This situation is similar to the development of many-valued probability theory starting from \L ukasiewicz logic: while every probability of an MV-algebra induces a unique ``classical'' probability \cite{Kroupa:Archivy}, the more general concept of MV-probability forms a solid basis for studying a number of stochastic phenomena, betting games among them; cf. \cite[Chapter 1]{Mundici11}.
\end{remark}

It is useful to introduce a weaker notion of playability. To this end, we need this preparatory result.
\begin{lemma}\label{lem:playmonreg}
	If $E$ is  a playable $\lucas_n$-valued effectivity function, then $E$ is coalition monotonic and regular.
\end{lemma}
\begin{proof}
	We proceed by contradiction to prove that $E$ is regular. Assume that there are $C\in \power N$,  $f\in \lucas_n^S$ and $i \in \{1, \ldots, n\}$ such that
	$
	\neg E(\overline{C}, \neg f) < \frac{i}{n} \leq E(C,f).
	$ Put $j=1-\frac{i-1}{n}$. Since $E$ is homogeneous, it follows that 
	$E(C,\tau_{\sfrac{i}{n}}(f))=1$ and $E(\overline{C}, \tau_{\sfrac{j}{n}}(\neg f))=1$. Thus we obtain by superadditivity 
	$E(N, \tau_{\sfrac{i}{n}}(f) \wedge \tau_{\sfrac{j}{n}}(\neg f))=1,$
	which contradicts safety  since $\tau_{\sfrac{i}{n}}(f) \wedge \tau_{\sfrac{j}{n}}(\neg f)=0$.
	
	We prove that $E$ is coalition monotonic. Let $C\subseteq C'  \in \power N$ and $f \in \lucas_n^S$. By applying superadditivity to $C_1=C$ and $C_2=C'\setminus C$ we obtain $E(C,f)\leq E(C',f)$, which is the desired result.
\end{proof}

\begin{definition}\label{defn:semi}
An $\lucas_n$-valued effectivity function $E\colon \power N \times \lucas_n^{S} \to \lucas_n$ is \emph{semi-playable} if $E(C,f)\leq E(C,g)$ for every $f\leq g \in \lucas_n^S$ and every coalition $C\neq N$, if $E$ has  liveness and  safety for coalitions $C\neq N$, and if it satisfies superadditivity for coalitions $C_1$ and $C_2$ such that $C_1 \cap C_2=\varnothing$ and $C_1 \cup C_2\neq N$.
\end{definition}

\begin{proposition}\label{prop:laz}
An  $\lucas_n$-valued effectivity function $E\colon \power N \times \lucas_n^{S} \to \lucas_n$ is playable if and only if it is semi-playable, homogeneous, regular and $N$-maximal.
\end{proposition}
\begin{proof}
The first implication follows from Lemma \ref{lem:playmonreg}. Conversely, assume that $E$ is semi-playable, homogeneous, regular and $N$-maximal. First we prove superadditivity. Let $C\in \power N$ with $C\neq N$. We have to verify that
$E(C,f)\wedge E(\overline{C},g)\leq E(N, f\wedge g).$
By way of contradiction, assume that there is $i\in\{1, \ldots, n\}$ such that
$E(N, f\wedge g) < \frac{i}{n} \leq E(C,f)\wedge E(\overline{C},g)$.
Since $E$ is homogeneous, we obtain $E(C, \tau_{i/n}(f))=1=E(\overline{C}, \tau_{i/n}(g))$, while, by $N$-maximality, $E(\varnothing, \neg\tau_{i/n}(f\wedge g))=1$.
It follows from superadditivity that $E(C, \tau_{i/n}(f) \wedge \neg \tau_{i/n}(f\wedge g))=1$, which is equivalent to  $E(C, \tau_{i/n}(f) \wedge \neg (\tau_{i/n}(f)\wedge \tau_{i/n}(g)))=1$,
since $\tau_{i/n}(f\wedge g)=\tau_{i/n}(f)\wedge \tau_{i/n}(g)$. From  the fact that $\tau_{i/n}(f)$ and $\tau_{i/n}(g)$ belong to the Boolean skeleton of $\lucas_n^S$ we deduce  
$
E(C, \tau_{i/n}(f) \wedge \neg \tau_{i/n}(g))=1.
$
We conclude that $E(C, \neg \tau_{i/n}(g))=1$ by outcome monotonicity  and finally that $E(\overline{C},  \tau_{i/n}(g))=0$ by regularity, which is the desired contradiction.

It is easy to check that the liveness and safety conditions are satisfied for $C=N$. Moreover, $E$ is $N$-maximal and homogeneous by assumption. It remains to prove that $E$ is outcome monotonic. If $f\leq g \in \lucas_n^S$, we obtain successively 
\[
E(N,f)\leq \neg E(\varnothing, \neg f)\leq \neg E(\varnothing, \neg g)\leq E(N,g)
\]
where the first inequality is obtained by regularity, the second by monotonicity and the third by $N$-maximality.
\end{proof}

\section{$\lucas_n$-valued modal language and semantics for effectivity functions}\label{sec:coal}
In this section we build a many-valued modal logic in the spirit of \cite{Pauly2002, Goranko2013} that captures the properties of (truly) playable $\lucas_n$-valued effectivity functions.

\subsection{Neighborhood semantics for playable $\lucas_n$-valued effectivity functions}\label{sec:neigh}
Let $\lang$ be the language  $\{\iimplies, \neg, 1\}\cup\{[C] \mid C \in \power N\}$ where $\iimplies$, $\neg$, $1$ are binary, unary and constant, respectively, and $[C]$ is a unary modality for every $C \in \power N$. The~set $\Form_\lang$  of formulas is defined inductively from the  countably infinite set $\Prop$ of propositional variables by the following rules:
\[
\phi::= 1\ \vert \ p \ \vert \ \phi \iimplies \phi \ \vert \ \neg \phi \ \vert \ [C]\phi \ 
\]
where $p \in \Prop$ and $C \in \power N$. We use $0$ as an abbreviation of $\neg 1$. The intended reading of the formula $[C]\phi$ is  `coalition $C$ can enforce $\phi$ '. In the language $\lang$ we also use the standard abbreviations for defined connectives in \L ukasiewicz logic; see Appendix \ref{appendix:MV}.

We introduce a semantics for $\lang$ which is based on a class of action models called \emph{$\lucas_{n}$-frames}. Such frames are $\lucas_{n}$-valued extensions of coalition frames introduced  in \cite{Pauly2002}. The coalition frames are a very general model of interaction in which an effectivity function is associated with each outcome state in $S$. Under the assumption of true playability, this is equivalent to specifying a game form for every outcome state in $S$.
\begin{definition}
An \emph{$\lucas_n$-frame} is a tuple $\framme{F}=(S,E)$, where $S$ is a non-empty set of outcome states and $E$ is a mapping sending each outcome state $u\in S$ to an~$\lucas_{n}$-valued effectivity function $E(u)\colon \power N \times \lucas_n^{S} \to \lucas_n$.
A tuple  $\model{M}=(\framme{F},\Val)$ is an~\emph{$\lucas_n$-model (based on $\framme{F}$)} if $\framme{F}=(S,E)$ is an $\lucas_n$-frame and $\Val\colon S\times \Prop \to \lucas_n$. 
\end{definition}

We use the \L ukasiewicz interpretations of the connectives $\neg, \iimplies, 1$ in $\lucas_{n}$; see Appendix \ref{appendix:MV}. For every $\lucas_n$-model $\model{M}$, the valuation map $\Val$ is extended inductively to $S\times\Form_\lang$ by setting
\begin{gather}
\label{eqn:val01}\Val(u, \phi \iimplies \psi) =\Val(u, \phi)\iimplies \Val(u, \psi),\\ 
\label{eqn:val02}\Val(u, \neg \psi) =\neg\Val(u,\psi),\\
\label{eqn:val03}\Val(u,1) =1,\\
\label{eqn:val}\Val(u, [C]\phi)=E(u)\big(C,\Val(-, \phi)\big),
\end{gather}
for every $C \in \power N$ and every $\phi,\psi \in \Form_\lang$. We use the standard notation and terminology. We say that a formula $\phi$ is \emph{true in $\model{M}=(\framme{F}, \Val)$} and write $\model{M}\models \phi $ if $\Val(u, \phi)=1$ for every  $u \in S$. A formula $\phi$ is \emph{valid} in an $\lucas_n$-frame $\framme{F}$ if it is true in every $\lucas_n$-valued  model based on $\framme{F}$.

In order to proceed further, we need to generalize  the technique of filtration \cite[Chapter 7.5]{Chellas1980} for neighborhood models.

\begin{definition}\label{defn:filt}
Let $\model{M}=(S,E, \Val)$ be an $\lucas_n$-valued   model and $\Gamma$ be a set of formulas closed under subformulas and the unary connectives $\tau_\oplus\colon \phi \mapsto \phi \oplus \phi$ and $\tau_\odot\colon \phi \mapsto \phi \odot \phi$. Consider the equivalence relation $\equiv_\Gamma$  defined on $S$ by
\[
u \equiv_\Gamma v \quad \text{if} \quad \forall \phi \in \Gamma \ \Val(u, \phi)=\Val(v,\phi).
\]
We denote by $|S|$ the set of equivalence classes $|u|$ for $\equiv_\Gamma$. An $\lucas_n$-model $\model{M}^*=(|S|, E^*, \Val^*)$ is a $\Gamma$-\emph{filtration} of $\model{M}$ if the following conditions are satisfied for every $u\in S$:
\begin{enumerate}
\item\label{it:filt01} $\Val^*(|u|,p)=\Val(u,p)$ for every $p \in \Prop \cap \Gamma$,
\item\label{it:filt02} $E(u)(C,\Val(-, \phi))=E^*(|u|)(C,|\Val(-, \phi)|)$ for every $C\in \power N$ and $\phi\in \Gamma$,
\end{enumerate} 
where the map $|\Val(-, \phi)| \colon |S| \to \lucas_n$ is defined by $|\Val(|u|, \phi)|=\Val(u,\phi)$.
\end{definition}

\begin{lemma}\label{lem:filt}
Let $\model{M}=(S,E, \Val)$ be an $\lucas_n$-valued  model and $\Gamma$ be a set of formulas closed under subformulas and the  connectives $\tau_\oplus$ and $\tau_\odot$. If $\model{M}^*=(|S|, E^*, \Val^*)$ is a $\Gamma$-filtration of $\model{M}$, then 
\begin{equation}\label{eqn:filt}
\Val(u, \phi)=\Val^*(|u|, \phi)
\end{equation}
 for every $\phi \in \Gamma$ and every $u\in S$. 
\end{lemma}
\begin{proof}
Note that identity \eqref{eqn:filt} is equivalent to $\Val^*(-,\phi)=|\Val(-,\phi)|$.
The proof is a standard  induction argument on the length of $\phi\in \Gamma$. We  consider only the case where $\phi=[C]\psi \in \Gamma$ for $C \in \power N$. By the definition of $\Val^*$ and the induction hypothesis, we obtain \[\Val^*(|u|, [C]\psi)  = E^*(|u|)(C, \Val^*(-, \psi))=E^*(|u|)(C, |\Val(-,\psi)|).\] It follows from Definition \ref{defn:filt}(\ref{it:filt02}) that \[E^*(|u|)(C, |\Val(-,\psi)|)=E(u)(C,\Val(-,\psi))=\Val(u, [C]\psi).\qedhere\] 
\end{proof}

In what follows we focus on the relations between the language $\lang$ and the $\lucas_n$-frames in which the effectivity functions are (truly) playable.

\begin{definition}
An $\lucas_n$-frame $\framme{F}=(S,E)$ is said to be \emph{(truly) playable} if $E(u)$ is \emph{(truly) playable} for every $u\in S$. An $\lucas_n$-model $\model{M}$ is \emph{(truly) playable} if it is based on a (truly) playable $\lucas_n$-frame.
\end{definition}

Our first aim is to prove that, similarly as in the Boolean case \cite{Goranko2013}, there is no set of $\lang$-formulas that can define truly playable $\lucas_n$-frames inside the class of playable $\lucas_n$-frames. To this end, we show that for any playable $\lucas_n$-model $\model{M}$ and any formula $\phi$, there is a finite playable $\lucas_n$-model $\model{M}_\phi$ such that $\model{M}\models \phi$ if and only if $\model{M}_\phi \models \phi$. We use this property and the fact that finite  playable $\lucas_n$-models are truly playable to prove Proposition \ref{prop:imp}. The construction of $\model{M}_\phi$ is based on a refinement of filtration for playable $\lucas_n$-valued  models. We proceed in two steps. The next definition constitutes the first step in this direction.

\begin{definition}\label{defn:int}
Using the notation of Definition \ref{defn:filt}, an $\lucas_n$-valued  model $\model{M}^*=(|S|, E^*, \Val^*)$ is an \emph{intermediate $\Gamma$-filtration of} a playable $\lucas_n$-model $\model{M}=(S, E, \Val)$ if $\model{M}^*$ is a $\Gamma$-filtration of $\model{M}$ that satisfies 
\begin{gather}
E^*(|u|)(C,f)=\max\{E(u)(C,\Val(-,\phi))\mid \phi \in \Gamma \text{ and } |\Val(-,\phi)|\leq f\},\label{eqn:int01}\\
E^*(|u|)(N,f)=\neg E^*(|u|)(\varnothing, \neg f),\label{eqn:int02}
\end{gather}
for every proper coalition $C\in \power N$ and every $f\in  \lucas_n^{|S|}$.
\end{definition}

Observe that since we have assumed playability ($N$-maximality, in particular) of $\model{M}$ in Definition \ref{defn:int}, an intermediate $\Gamma$-filtration of $\model{M}$ is indeed a $\Gamma$-filtration in the sense of Definition \ref{defn:filt}.

For every formula $\mu$ we denote by $\Cl(\mu)$ the closure of the set of subformulas of $\mu$ for the connectives $\neg$ and $\iimplies$.  The next lemma shows that an intermediate $\Cl(\mu)$-filtration is an intermediate step in the construction of a playable $\Cl(\mu)$-filtration of a playable $\lucas_n$-model. Recall that by $\br{E(u)}$ we denote the function that is the restriction of the $\lucas_{n}$-valued effectivity function $E(u)\colon \power N\times \lucas_{n}^{S}\to \lucas_{n}$ to the domain $\power N\times \lucas_{1}^{S}$.

\begin{lemma} \label{lem:res}
Let $\mu\in \Form_\lang$ and $\model{M}=(S,E, \Val)$ be a playable $\lucas_n$-model. If $\model{M}^*=(|S|, E^*, \Val^*)$ is an intermediate $\Cl(\mu)$-filtration of $\model{M}$, then the Boolean effectivity function $E^{*}(|u|)^\sharp\colon\power N \times \lucas_1^S \to \lucas_1$ is playable for every $u\in S$.
\end{lemma}

\begin{proof}
By $n.\phi$ we denote the formula $\bigoplus_{i=1}^{n}\phi$. First, observe that if $u\in S$,  $f\in \lucas_1^{|S|}$ and $C\neq N$ is a coalition, then
\[
E^*(|u|)(C,f)=\max\{E(u)^\sharp(C,\Val(-,n.\phi)) \mid \phi \in \Gamma \text{ and } |\Val(-,n.\phi)|\leq f\},
\]
which shows that $E^{*}(|u|)^\sharp$ is a Boolean effectivity function. We prove that $E^{*}(|u|)^\sharp$ is regular and semi-playable (see Definition \ref{defn:semi}). It is straightforward to show that $E^{*}(|u|)^\sharp(C,-)$ is monotonic and has liveness and  safety for every coalition $C\neq N$. Moreover, $N$-maximality holds for $E^{*}(|u|)^\sharp$ according to \eqref{eqn:int02}.

Let us prove superadditivity for coalitions $C, D$ such that $C\cap D=\varnothing$ and $C\cup D \neq N$. If $\smash{f,g\in \lucas_1^{|S|}}$, then  $E^{*}(|u|)(C, f) \wedge E^*(|u|)(D,g)$ is  by definition equal to the maximum of the values
$E(u)(C, \Val(-,\psi)) \wedge E(u)(D, \Val(-,\rho))$, where $\psi$ and $\rho$ run through the elements of  $\Cl (\mu)$ satisfying $ |\Val(-,\psi)|\leq f$ and $|\Val(-,\rho)|\leq g$. By superadditivity of $E$ we have \[E(u)(C, \Val(-,\psi)) \wedge E(u)(D, \Val(-,\rho)) \leq E(u)(C\cup D, \Val(-,\psi \wedge \rho)),\] for every formula $\psi$ and $\rho$. Thus it follows from the definition of $E^*$ that
\[E^{*}(|u|)(C, f) \wedge E^*(|u|)(D,g)\leq E^*(|u|)(C\cup D,f\wedge g),\]
which is the desired result.

It remains to prove that for every $C\in \power N$ and every $f\in \lucas_1^{|S|}$ such that $E^{*}(|u|)(C,f)=1$, we have $E^*(|u|)(\overline{C},\neg f)=0$. By condition  \eqref{eqn:int02}, we may assume that $C\neq N$. Suppose for the sake of contradiction that  $E^*(|u|)(\overline{C},\neg f)=1$. By \eqref{eqn:int01} this means that there are some $\psi, \rho \in \Cl(\mu)$ such that $|\Val(-,\psi)|\leq f$ and $|\Val(-,\rho)|\leq \neg f$, and $E(u)(C,\Val(-,\psi))=E(u)(\overline{C},\Val(-,\rho))=1$. By superadditivity of $E$, we obtain $E(N, \Val(-, \psi \wedge \rho))=1$ with $\psi \wedge \rho \in \Cl(\mu)$ satisfying $|\Val(-,\psi \wedge \rho)| \leq f \wedge \neg f$. By \eqref{eqn:int01}, \eqref{eqn:int02}, Definition \ref{defn:filt} (\ref{it:filt02}), $N$-maximality of $E$  and the fact that $\smash{f\in \lucas_1^{|S|}}$, we obtain $E^*(|u|)(N, 0)=1$. This is a contradiction since $E^*(|u|)(N, 0)=\neg E^*(|u|)(\varnothing, 1)=0$ by \eqref{eqn:int02} and liveness of $E^*$ for the empty coalition.
\end{proof}

We combine Lemma \ref{lem:res} together with Lemma \ref{lem:new} to construct $\Cl(\mu)$-filtrations that preserve playability.

\begin{proposition}\label{prop:filt}
If $\mu\in \Form_\lang$  and $\model{M}=(S,E, \Val)$ is a playable $\lucas_n$-model, then there is a playable $\Cl(\mu)$-filtration $\model{M}^+=(|S|, E^+, \Val^+)$ of $\model{M}$.
\end{proposition}
\begin{proof}
Consider the intermediate $\Cl(\mu)$-filtration $\model{M}^*=(|S|, E^*, \Val^*)$ of $\model{M}$. By Lemma \ref{lem:res}, the Boolean effectivity function $\smash{E^{*}(|u|)^\sharp\colon\power N \times \lucas_1^{|S|} \to \lucas_1}$ is playable for every $u\in S$.  By Lemma \ref{lem:new}, the map $\smash{E^+(|u|)\colon\power N \times \lucas_n^{|S|} \to \lucas_n}$ defined by 
\[
E^+(|u|)(C, f)=\max\left\{\tfrac in\in\lucas_{n} \mid E^{*}(|u|)^\sharp(C,\tau_{\sfrac in}(f))=1\right\}
\]
is also playable. We prove that $\model{M}^*:=(|S|, E^+, \Val^*)$ is a $\Cl(\mu)$-filtration of $\model{M}$. It suffices to check that $\model{M}^+$ satisfies condition (\ref{it:filt02}) of Definition \ref{defn:filt}. Let $u \in S$, $C \in \power N$, $\phi \in \Cl(\mu)$ and $i\in \{1, \ldots,n\}$. We obtain by definition of $E^+$ that $E^+(|u|)(C, |\Val(-,\phi)|)\geq \sfrac{i}{n}$ if and only if $E^{*}(|u|)^\sharp(C, |\Val(-,\tau_{\sfrac{i}{n}}(\phi))|)=1$. Since $\model{M}^*$ is an intermediate $\Cl(\mu)$-filtration of $\model{M}$, the last identity is in turn equivalent to $E(u)(C,\Val(-,\tau_{\sfrac{i}{n}}(\phi))=1$. Finally, this gives $E(u)(C, \Val(-,\phi))\geq \sfrac{i}{n}$.
\end{proof}

The next result shows that the gain of expressive power induced by the many-valued nature of $\lang$ and of its associated semantics is not enough to single out those playable models that are truly playable.

\begin{proposition}\label{prop:imp}
There is no set $\Lambda$ of $\lang$-formulas such that a playable $\lucas_n$-frame $\framme{F}$ is truly playable if and only if  every formula of $\Lambda$ is valid in $\framme{F}$.
\end{proposition}
\begin{proof}
Assume that there exists a set $\Lambda$ of $\lang$-formulas such that a playable $\lucas_n$-frame $\framme{F}$  is truly playable if and only if every formula of $\Lambda$ is valid in~$\framme{F}$.
Let $\framme{F}$ be a playable $\lucas_n$-frame which is not truly playable. The existence of such $\framme{F}$ is a consequence of Lemma \ref{lem:new} applied to the effectivity function $E$ defined in \cite[Proposition 4]{Goranko2013}. For every $\phi \in \Lambda$ and every model $\model{M}$ based on $\framme{F}$, Proposition \ref{prop:filt} provides a playable $\Cl(\phi)$-filtration $\model{M}^+$. Since $\model{M}^+$ has a finite set of outcome states, it is truly playable. It follows from the definition of $\Lambda$ that $\model{M}^+\models \phi$ and from Lemma \ref{lem:filt} that $\model{M}\models \phi$. We have proved that every formula of $\Lambda$ is true in every model based on $\framme{F}$ and we conclude that $\framme{F}$ is truly playable, which is the desired contradiction.
\end{proof}

\subsection{$\lucas_n$-valued playable logic for finite playable $\lucas_n$-frames}

Proposition \ref{prop:imp} says that $\lang$ is not adequate for capturing the properties of $\lucas_n$-valued effectivity functions associated with game forms. Indeed, this language is not even expressive enough to distinguish between the playable and the truly playable $\lucas_n$-frames. Nevertheless, when the set of outcome states $S$ is finite, every playable $\lucas_n$-valued effectivity function is truly playable and it turns out that playability can be encoded by $\lang$-formulas; see our completeness result, Theorem \ref{thm:comp}. We start with axiomatizing the properties of playable $\lucas_n$-valued effectivity functions.

\begin{definition}\label{defn:log}
An \emph{$\lucas_n$-valued playable logic} is a subset $\logic{L}$ of $\Form_{\lang}$ which is closed under Modus Ponens, Uniform Substitution and Monotonicity (if $\phi \iimplies \psi \in \logic{L}$, then $[C]\phi \iimplies [C]\psi \in \logic{L}$ for every $C\in \power N$) and that contains an axiomatic base of \L ukasiewicz $(n+1)$-valued logic (see \cite{Grigolia77} or \cite[Section 8.5]{Cignoli2000}) together with the following axioms:
\begin{framed}
\noindent
\textbf{The axioms of $\lucas_n$-valued playable logic}
\begin{enumerate}
\begin{multicols}{2}
\item\label{eqn:ax01} $[C](p \odot p)  \leftrightarrow  [C] p \odot [C] p$,
\item\label{eqn:ax02} ${[C](p \oplus p)}  \leftrightarrow  [C] p \oplus [C] p$,
\item\label{eqn:ax03} $\neg[C]0$, 
\item\label{eqn:ax04} $([C] p \wedge [C']q ) \iimplies [C\cup C'] (p \wedge q)$,
\item\label{eqn:ax05} $[\varnothing] p  \iimplies \neg [N] \neg p,$
\end{multicols}
\end{enumerate}
for every $C, C' \in \power N$ such that $C\cap C'=\varnothing$. 
\end{framed}
\noindent
We denote by $\logic{P}_n$ the \emph{smallest   $\lucas_n$-valued playable logic}, that is, the intersection  of all the  $\lucas_n$-valued playable logics. We conform with common usage and we often write $ \vdash_{\logic{P}_n} \phi$ instead of $\phi \in \logic{P}_n$.

\begin{remark}\label{rem:proofsyst}
The use of the notation $ \vdash_{\logic{P}_n}$ is justified by the observation that  $\logic{P}_n$ can be equivalently introduced through a Hilbert style proof system. Indeed, it suffices to consider the Hilbert system whose axioms are the axioms (\ref{eqn:ax01})--(\ref{eqn:ax05}) above together with an axiomatic base of  \L ukasiewicz $(n+1)$-valued logic,  and whose inference rules are Modus Ponens, Uniform Substitution and Monotonicity. Clearly, a formula $\phi$ is a theorem in this system if and only if it belongs to $\logic{P}_n$. 
\end{remark}

\end{definition}

The axioms \eqref{eqn:ax01}--\eqref{eqn:ax05} together with the Monotonicity rule reflect the properties defining playability. In Remark \ref{rem:equiv} at the end of this section, we give equivalent and more intuitive axioms that can replace \eqref{eqn:ax01}--\eqref{eqn:ax02} in the axiomatization of $\logic{P}_n$. 

The following lemma can be proved by a standard induction argument.
\begin{lemma}\label{lem:ovx}
Let $\model{M}$ be a playable $\lucas_n$-model. If $\vdash_{\logic{P}_n } \phi$, then $\model{M} \models \phi$. 
\end{lemma}
\noindent
We will prove completeness of $\logic{P}_n$ with respect to the class of playable $\lucas_n$-models. Our proof is based on the construction of the canonical model.

\subsubsection{Construction of the canonical model}\label{sec:canon} 

Let us denote by $\free_{\logic{P}_n}$ the Lindenbaum-Tarski algebra of $\logic{P}_n$, that is,  the quotient of $\Form_\lang$ under the syntactic equivalence relation $\equiv$ defined by
\[
\phi  \equiv \psi \quad \mbox{ if } \quad    \vdash_{\logic{P}_n} \phi \iimplies   \psi \text{ and } \vdash_{\logic{P}_n} \psi \iimplies \phi,
\]
equipped with the operations  $1$, $\neg$, $\iimplies$ and $[C]$ defined as $1:=1/\equiv$, $\neg (\phi/\equiv):=\neg\phi/\equiv$, $ \phi/\equiv\, \iimplies \psi/\equiv\ :=(\phi \iimplies \psi)/\equiv$ and $[C](\phi/\equiv) :=[C]\phi/\equiv$, for every $C\in \power N$ and every $\phi, \psi \in \Form_\lang$. 
By abuse of notation, we denote the class $\phi/\equiv$ by $\phi$.

Since the logic $\logic{P}_n$ contains every tautology of \L ukasiewicz $(n+1)$-valued logic, the $\{\iimplies, \neg, 1\}$-reduct of  $\free_{\logic{P}_n}$ is an MV-algebra that belongs to the variety $\var{MV}_n$ generated by $\lucas_n$.

In the Boolean setting, one of the key ingredients of the construction of the canonical model is the ultrafilter theorem that allows us to separate by an ultrafilter any two different non-top elements of a Boolean algebra $\alg{B}$. We can rephrase this separation result using the bijective correspondence between the ultrafilters of $\alg{B}$ and the homomorphisms of $\alg{B}$ into the two-element Boolean algebra $\alg{2}$: for every $a\neq b \in \alg{B}\setminus\{1\}$, there is a homomorphism $u\colon \alg{B}\to \alg{2}$ such that $u(a)=1$ and $u(b)=0$. The variety $\var{MV}_n$ has an analogous property \cite{Cignoli2000}.
\begin{lemma}\label{lem:seppro}
Let $\alg{A}\in \var{MV}_n$. For every $a\neq b$ in $\alg{A}\setminus\{1\}$, there is a $\{\neg, \iimplies, 1\}$-homomorphism $u\colon \alg{A}\to \lucas_n$ such that $u(a)=1$ and $u(b)\neq 1$.
\end{lemma}
\noindent
This separation property explains our choice of the set $W^c$ of $\{\neg, \iimplies, 1\}$-homomorphisms from $\free_{\logic{P}_n}$ to $\lucas_n$ as the universe of the canonical model of $\logic{P}_n$.

We will use the following technique to associate an $\lucas_n$-valued effectivity function $E^c(u)$ with every $u \in W^c$. For each $i\in \{1, \ldots, n\}$ we will define a subset $P_i$ of $\power N \times W^c \times \smash{\lucas_n^{ W^c}}$, such that $P_1\supseteq \dotsb \supseteq  P_{n}$. Then we will safely set
\[
E^{c}(u)(C,f):=\max \left\{\tfrac{i}{n}\in\lucas_{n} \mid (C,u,f) \in P_i\right\}.
\] 

\begin{definition}
For every $i \in\{1, \ldots, n\}$ let $P_i$ be the subset of $\power N \times W^c \times \lucas_n^{W^c}$ defined by
\begin{equation}\label{eqn:trc}
 P_i=\left\{(C,u,f)\mid \exists \phi \big( u([C]\phi)\geq \tfrac{i}{n} \ \text{ and }  \ \forall v\in W^{c} \big(v(\phi)\geq \tfrac{i}{n} \implies f(v)\geq \tfrac{i}{n}\big) \big)\right\}.
\end{equation}
We use the convention $P_0=\power N \times W^c \times \smash{\lucas_n^{ W^c}}$.
\end{definition}

\begin{lemma}\label{lem:tech}
The inclusion $P_i \subseteq P_{i-1}$ holds for each $i\in\{1, \ldots, n\}$.
\end{lemma}
\begin{proof}  Assume that  $\phi$ satisfies the condition defining $P_i$ in \eqref{eqn:trc} for $C$, $f$, $u$ and $i>0$ and put $\rho=\tau_{\sfrac{i}{n}}(\phi)$. Then $u([C]\rho)=\tau_{\sfrac{i}{n}}\big(u([C]\phi)\big)=1\geq \frac{i-1}{n}$. Moreover, if~$v(\rho)\geq \frac{i-1}{n}$, then $v(\rho)=1$ since $\rho$ belongs to the Boolean skeleton of $\free_{\logic{P}_n}$. Therefore $v(\phi)\geq \frac{i}{n}$, which gives $f(v)\geq \frac{i}{n}\geq \frac{i-1}{n}$.
\end{proof}

\begin{definition}\label{defn:canon}
The \emph{canonical model} of $\logic{P}_n$ is the $\lucas_n$-model $\model{M}=(\framme{F}, \Val^c)$ with $\framme{F}=(W^c,  E^c)$  where $E^c(u)(C, f)$ is defined for every $u\in W^c$, every $C\in \power N$ and every $f\in \smash{\lucas_n^{ W^c}}$ by
\begin{equation} \label{eqn:jkl}
E^c(u)\big(C, f)=\left\{\begin{array}{lr}
\max \left \{\tfrac{i}{n}\in\lucas_{n} \mid (C,u,f) \in P_i\right\} & \text{ if } C\neq N,\\ 
\neg E^c(u)(\varnothing, \neg f) & \text{ if } C=N,
\end{array}\right.
\end{equation} 
and  where $\Val^c$ is defined by
\begin{equation}\label{eqn:valc}
\Val^c(u,p)=u(p), 
\end{equation}
for every $p \in \Prop$ and $u \in W^c$.
\end{definition}

In particular, for every $(C,u,f)\in \power N \times W^c \times \smash{\lucas_n^{ W^c}}$ with $C\neq N$, we have
\begin{equation}\label{eqn:nap}
E^c(u)\big(C, f) \geq \frac{i}{n}\quad \text{ if and only if }\quad  (C, u, f) \in P_i. 
\end{equation}
 The next proposition shows that the identity \eqref{eqn:valc} remains true in the canonical model after replacing $p$ by any formula $\mu \in \Form_\lang$.

\begin{proposition}[Truth Lemma]\label{prop:tru}
The canonical model $(W^c,  E^c,\Val^{c})$ of $\logic{P}_n$ satisfies $\Val^c(u, \mu)=u(\mu)$ for every $\mu \in \Form_\lang$ and every $u \in W^c$.
\end{proposition}
\begin{proof} 
We proceed by induction on the number of connectives in $\mu$. If $\mu\in \Prop$  or $\mu \in \{1,\neg\psi,\psi \iimplies \rho\}$, then the result follows immediately from (\ref{eqn:val01})~--~(\ref{eqn:val03}).

Let $\mu=[C]\psi$ for some $\psi \in \Form_\lang$ and $C\in \power N\setminus \{N\}$. We will prove that for any $u\in W^c$ and  $i\leq n$,
\begin{equation}\label{eq:truthlemma}
E^c(u)\big(C,\Val^c(-, \psi)\big)\geq \frac{i}{n}\quad \text{ if and only if }\quad u([C]\psi)\geq \frac{i}{n}.
\end{equation}
\noindent
First, assume $E^c(u)\big(C,\Val^c(-, \psi)\big)\geq \frac{i}{n}$. Then, by \eqref{eqn:nap} and \eqref{eqn:trc}, there is $\rho \in \Form_\lang$ such that $u([C]\rho)\geq \frac{i}{n}$ and $\Val^c(v, \psi)\geq \frac{i}{n}$ for any $v\in W^c$ satisfying $v(\rho) \geq \frac{i}{n}$. By the induction hypothesis, this means that for every $v \in W^c$ with $v(\rho)\geq \frac{i}{n}$, we have $v(\psi)\geq \frac{i}{n}$.
It follows that $v\big(\tau_{\sfrac{i}{n}}(\rho) \iimplies \tau_{\sfrac{i}{n}}(\psi)\big)=1$ for every $v \in W^c$. This yields $ \vdash_{\logic{P}_n} \tau_{\sfrac{i}{n}}(\rho) \iimplies \tau_{\sfrac{i}{n}}(\psi) $ since the $\{\iimplies, \neg, 1\}$-reduct of $\free_{\logic{P}_n}$ has the separation property (Lemma \ref{lem:seppro}). As $\logic{P}_n$ is closed under Monotonicity, we obtain \[  \vdash_{\logic{P}_n} [C] \tau_{i/n}(\rho) \iimplies [C] \tau_{i/n}(\psi).\] By axioms (\ref{eqn:ax01}) and (\ref{eqn:ax02}) of $\logic{P}_n$ (Definition \ref{defn:log}) and Uniform Substitution, this is equivalent to \[\vdash_{\logic{P}_n } \tau_{\sfrac{i}{n}}([C]\rho) \iimplies  \tau_{\sfrac{i}{n}}([C]\psi).\] Hence, if $v\in W^c$ and $v([C]\rho)\geq \frac{i}{n}$, then $v([C]\psi)\geq \frac{i}{n}$. We can thus conclude that $u([C]\psi)\geq \frac{i}{n}$.

Conversely, let $u([C]\psi)\geq \frac{i}{n}$. We obtain $(C,u,\Val(-,\psi))\in P_i$ by the induction hypothesis and by considering $\phi=\psi$ in the definition (\ref{eqn:trc}) of $P_i$.  This proves \eqref{eq:truthlemma}.

Finally, assume $\mu= [N]\psi$ for some $\psi \in \Form_\lang$. We will prove \[E^c(u)(N, \Val(-,\psi))=u([N]\psi).\] Indeed, on the one hand we obtain \[E^c(u)(N, \Val(-,\psi))  =  \neg E^c(u)(\varnothing, \Val(-, \neg \psi))\]  by \eqref{eqn:jkl}, which is in turn equal to $ u(\neg [\varnothing]\neg \psi)$ by the induction hypothesis and the first part of this proof. Axiom
\eqref{eqn:ax05} of $\logic{P}_n$ yields $E^c(u)(N, \Val(-,\psi))\leq u([N]\psi)$.

To prove the converse inequality, let us assume for the sake of contradiction that there is  $i\in\{1, \ldots, n\}$ such that 
$E^c(u)(N, \Val(-,\psi))<\frac{i}{n}= u([N]\psi)$.
Therefore $u([N]\tau_{i/n}(\psi))=1$, while $E^c(u)(\varnothing, \Val(-,\neg\tau_{i/n}(\psi))=1$ by \eqref{eqn:jkl}. It follows from this identity that $u([\varnothing]\neg \tau_{\sfrac{i}{n}}(\psi))=1$ by induction hypothesis and the first part of the proof. From axiom \eqref{eqn:ax06} applied with $C=N$ and $C'=\varnothing$ we get $u([N]\big(\tau_{\sfrac{i}{n}}(\psi) \wedge \neg \tau_{\sfrac{i}{n}}(\psi))\big)=1$; however, this is in contradiction with axiom \eqref{eqn:ax04} since $\tau_{\sfrac{i}{n}}(\psi) \wedge \neg \tau_{\sfrac{i}{n}}(\psi)=0$. 
\end{proof}

\subsubsection{Completeness result for $\logic{P}_n$}
In order to use the canonical model for the proof of completeness of $\logic{P}_n$ with respect to the class of the playable $\lucas_n$-models, we need the following result.

\begin{lemma}\label{lem:goo}
The canonical model of $\logic{P}_n$ is a playable $\lucas_n$-frame.
\end{lemma}
\begin{proof}
Let $u \in W^c$. It suffices to prove that $E^c(u)$ is semi-playable,  homogeneous, $N$-maximal and regular. It is easily checked that $E^c(u)(C, -)$ is monotonic for every coalition $C\neq N$.  The property of $N$-maximality is obtained by  \eqref{eqn:jkl}.  

For homogeneity, let $C\neq N \in \power N$ and $f \in \lucas_n^{W^c}$. We  will prove that for any $i\in \{1,\ldots, n\}$,
\begin{equation}\label{eq:lemplay}
E^c(u)(C,f\oplus f)\geq \frac{i}{n}\quad \text{ if and only if }\quad E^c(u)(C,f) \oplus E^c(u)(C,f)\geq \frac{i}{n}.
\end{equation}

First, assume $E^c(u)(C,f) \oplus E^c(u)(C,f)\geq \tfrac{i}{n}$ or, equivalently, $E^c(u)(C,f)\geq \tfrac{i}{2n}$. By the definition of $E^c$, there is a formula $\rho$ such that $u([C]\rho)\geq \frac{i}{2n}$ and $f(v)\geq \frac{i}{2n}$ for every $v$ satisfying $v(\rho)\geq\frac{i}{2n}$. On the one hand, by axioms \eqref{eqn:ax01} and \eqref{eqn:ax02} of $\logic{P}_n$ and considering $\phi=\tau_{\lceil \sfrac{i}{2n}\rceil}(\rho)$, we get
$
u([C]\phi)=\tau_{\lceil \sfrac{i}{2n} \rceil}(u([C]\rho))=1\geq \frac{i}{n}.
$
On the other hand, if $v\in W^c$ is such that $v(\phi)\geq \frac{i}{n}$, then $v(\phi)=1$ since $\phi$ is an idempotent element of $\free_{\logic{P}_n}$. Therefore $v(\rho)\geq \frac{i}{2n}$. This implies $f(v)\geq \frac{i}{2n}$, or, equivalently, $(f\oplus f)(v)\geq \frac{i}{n}$. We conclude that $E^c(u)(C,f\oplus f)\geq \frac{i}{n}$.

Conversely, assume  $E^c(u)(C,f\oplus f)\geq \frac{i}{n}$ for some $i>0$. The definition of $E^c$ yields a formula $\rho$ such that $u([C]\rho)\geq \frac{i}{n}$ and $f(v)\geq \frac{i}{2n}$  for any $v\in W^c$ with $v(\rho)\geq \frac{i}{n}$. By considering $\phi=\tau_{i/n}(\rho)$, we obtain on the one hand that 
$
u([C]\phi)=\tau_{\sfrac{i}{n}}\big(u([C]\rho)\big)=1\geq \frac{i}{2n}.
$
On the other hand, if $v\in W^c$ is such that $v(\phi)\geq\frac{i}{2n}$, then $v(\phi)=1$,
which means $v(\rho)\geq \frac{i}{n}$ so that $f(v)\geq \frac{i}{2n}$. We have proved that $E^c(u)(C, f)\geq \frac{i}{2n}$ or, equivalently, $E^c(u)(C, f)\oplus E^c(u)(C, f)\geq \frac{i}{n}$. This finishes the proof of \eqref{eq:lemplay}.

Analogously, we can show that
\[
E^c(u)(C,f\odot f)\geq \frac{i}{n}\quad \text{ if and only if }\quad E^c(u)(C,f) \odot E^c(u)(C,f)\geq \frac{i}{n}.
\]
Employing $N$-maximality and the first part of the proof, it is easy to prove
that for every $f\in \lucas_n^{W^c}$, we have $E^c(u)(N, f\oplus f)=E^c(u)(N,f) \oplus E^c(u)(N,f)$ and $E^c(u)(N, f\odot f)=E^c(u)(N,f) \odot E^c(u)(N,f)$. The function $E^c(u)$ is hence homogeneous.

Let us prove that $E^c(u)$ has  safety  for $C\neq N$. By way of contradiction, assume that $E^c(u)(C, 0)\geq \frac{1}{n}$. There is a formula $\phi$ such that $u([C]\phi)\geq \frac{1}{n}$ and $v(\phi)=0$ for every $v\in W^c$. We deduce that $\vdash_{\logic{P}_n } \phi \iimplies 0$ and hence $\vdash_{\logic{P}_n } [C]\phi \iimplies [C]0$ by Monotonicity. It follows that $\frac{1}{n}\leq u([C]\phi)\leq u([C]0)=0$, a contradiction.

To prove that $E^c(u)$ has  liveness for every $C\neq N$, it suffices to consider $\phi=1$ in \eqref{eqn:trc} in order to show $(C,u, 1) \in P_1$.

We have to prove coalition monotonicity for $C_1$ and $C_2$ such that $C_1 \cap C_2=\varnothing$ and $C_1\cup C_2\neq N$. It is enough to prove the following:  if $E^c(u)(C_1,f_1)\wedge E^c(u)(C_2,f_2)\geq \frac{i}{n}$ for some $i \in \{1, \ldots, n\}$, then $E^c(u)(C_1\cup C_2, f_1\wedge f_2)\geq \frac{i}{n}$. Let $\ell \in\{1,2\}$ and denote by $\phi_\ell$ a formula such that $u([C_\ell] \phi_\ell)\geq \frac{i}{n}$ and $f_\ell(v)\geq \frac{i}{n}$, for every $v$ satisfying $v(\phi_\ell)\geq \frac{i}{n}$. Thus we can consider $\phi=\phi_1\wedge \phi_2$ in \eqref{eqn:trc} to get  $(C_1\cup C_2, u, f_1 \wedge f_2) \in P_{i}$.

It remains to check that $E^c(u)$ is regular. By \eqref{eqn:jkl}, it suffices to prove that it is $C$-regular for every $C\neq N$. For the sake of contradiction, assume that there exists $i\in\{1, \ldots, n\}$ such that 
$
\neg E^c(u)(\overline{C}, \neg f) < \frac{i}{n} \leq   E^c(u)({C},  f).
$
It follows that $E^c(u)({C},  \tau_{\sfrac{i}{n}}(f))=1$, while $E^c(u)(\overline{C}, \tau_{\sfrac{j}{n}}(\neg f))=1$ for $\frac{j}{n}=1-\frac{i-1}{n}$. By superadditivity, we obtain
$
E^c(u)(N, \tau_{\sfrac{i}{n}}(f) \wedge  \tau_{\sfrac{j}{n}}(\neg f) )=1,
$
which is a contradiction since $\tau_{\sfrac{i}{n}}(f) \wedge  \tau_{\sfrac{j}{n}}(\neg f)$ is the constant map $0$.
\end{proof}

\begin{theorem}[Completeness of $\logic{P}_n$]\label{thm:comp}  For any $\phi \in \Form_\lang$, the following assertions are equivalent:
\begin{enumerate}
\item\label{it:pqa01} $\vdash_{\logic{P}_n} \phi$.
\item\label{it:pqa02} $\phi$ is true in every playable $\lucas_n$-model.
\item\label{it:pqa03} $\phi$ is true in every finite playable $\lucas_n$-model.
\item\label{it:pqa04} $\phi$ is true in every truly playable $\lucas_n$-model.
\end{enumerate}
\end{theorem}
\begin{proof}
The implication (\ref{it:pqa01}) $\implies$ (\ref{it:pqa02}) is the content of Lemma \ref{lem:ovx}. The equivalences (\ref{it:pqa02}) $\iff$ (\ref{it:pqa03}) and (\ref{it:pqa03}) $\iff$ (\ref{it:pqa04}) follow from  Proposition \ref{prop:filt} and Lemma \ref{lem:filt}. Finally, it remains to argue for the implication (\ref{it:pqa02}) $\implies$ (\ref{it:pqa01}). We know by Lemma \ref{lem:goo} that $\model{M}^c$ is a playable $\lucas_n$-model. According to Proposition \ref{prop:tru}, $\model{M}^c \models \phi$ means that the class of $\phi$ is equal to $1$ in $\free_{\logic{P}_n}$, or, equivalently, $ \vdash_{\logic{P}_n} \phi$.
\end{proof}

\begin{remark}\label{rem:equiv}	
We can use the formulas $\tau_{\sfrac{i}{n}}(p)$ to replace axioms \eqref{eqn:ax01} and \eqref{eqn:ax02}  of $\logic{P}_n$ (Definition \ref{defn:log}) by a family of axioms, which are easier to understand. Indeed, put
\begin{equation*}
\begin{split}
A &=\{[C](p\star p)\leftrightarrow ([C] p \star [C] p)\mid \star\in\{\odot, \oplus\}, C\in \power N\},\\
B &=\{[C]\tau_{i/n}(p)\leftrightarrow\tau_{i/n}([C]p)\mid i \in\{1, \ldots, n\}, C\in \power N\}.
\end{split}
\end{equation*}
It follows from the definition of an $\lucas_n$-valued playable  logic that $B\subseteq \logic{P}_n$. A careful analysis of the proofs of Lemma \ref{lem:tech}, Proposition \ref{prop:tru}, Lemma \ref{lem:goo} and Theorem \ref{thm:comp} shows that we have only used the axioms in $A$ in the form of substitutions in formulas of $B$. Denote by $\logic{P}'_n$ the smallest set of formulas that contains an~axiomatic base of \L ukasiewicz $(n+1)$-valued logic, the set $B$, the axioms (\ref{eqn:ax03})--(\ref{eqn:ax05}) of Definition \ref{defn:log}, and that is closed under Modus Ponens, Uniform Substitution and Monotonicity. It follows from the previous observation  that for  any $\phi \in \Form_\lang$ we have $ \vdash_{\logic{P}'_n} \phi$ if and only if $\model{M}\models \phi$ for every  playable $\lucas_n$-model. Thus $\logic{P}'_n=\logic{P}_n$.

Thus the set of axioms $A$ can be equivalently replaced by $B$. Hence, the content of axioms \eqref{eqn:ax01}--\eqref{eqn:ax02} of $\logic{P}_{n}$ can be rephrased as follow.
\begin{quotation}
 For any $i\leq n$, the following two assertions are equivalent:
 \begin{itemize}
 \item The truth value of the statement `coalition $C$ can enforce $\phi$' is at least $\frac{i}{n}$.
 \item Coalition $C$ can enforce an outcome state in which the truth value of $\phi$ is at least $\frac{i}{n}$.
 \end{itemize}
\end{quotation}

\end{remark}

\subsection{$\lucas_n$-valued truly playable logics for  truly playable enriched $\lucas_n$-frames}
Theorem \ref{thm:comp} says that $\logic{P}_n$ is the logic of playable rather than truly playable effectivity functions. Moreover, by Proposition \ref{prop:imp} there is no axiomatization of truly playable effectivity functions in the language $\lang$. Thus the presented many-valued approach is a faithful generalization of the Boolean framework; see \cite{Goranko2013}. In fact the authors of \cite{Goranko2013} go beyond this limitation in the Boolean setting by adding a~new connective to $\lang$ and by enriching the neighborhood semantics with a Kripke relation. We follow this idea by designing the modal equivalent of truly playable $\lucas_n$-valued effectivity functions.

Let $\lang^+$ be the language $\lang\cup \{[\Oo]\}$ where $[\Oo]$ is unary.  The set $\Form_{\lang^+}$  of formulas is defined inductively from the countably infinite set $\Prop$ of propositional variables by the following rules:
\[
\phi::= 1\ \vert \ p  \ \vert \ \phi \iimplies \phi \ \vert \ \neg \phi \ \vert \ [C]\phi \ \vert\  [\Oo]\phi
\]
where $p \in \Prop$ and $C \in \power N$. 

In order to interpret $\lang^+$-formulas, we enrich the $\lucas_n$-frames with a~binary relation.
\begin{definition}\label{defn:lplus}
A tuple $\framme{F}=(S,E,R)$ is an \emph{enriched $\lucas_n$-frame} if $(S,E)$ is  an $\lucas_n$-frame and $R \subseteq S\times S$. We say that  $\framme{F}=(S,E,R)$ is \emph{standard} if $R=\{(u,v)\mid E(u)(\varnothing, \neg \chi_{\{v\}})=0\}$.


A tuple  $\model{M}=(S,E,R, \Val)$ is an \emph{enriched $\lucas_n$-model (based on $(S,E,R)$)} if $(S,E,R)$ is an enriched $\lucas_n$-frame and $\Val\colon S\times \Prop \to \lucas_n$.
An enriched $\lucas_n$-frame $\framme{F}=(S,E,R)$ or an enriched $\lucas_n$-model $\model{M}=(S,E,R,\Val)$ is called \emph{playable} (\emph{truly playable}, respectively) if $(S,E)$ is a playable (truly playable, respectively) $\lucas_n$-frame.
\end{definition}

In an enriched $\lucas_n$-model,  the valuation map $\Val$ is extended inductively to $S\times\Form_{\lang^+}$ by using rules \eqref{eqn:val01}~--~\eqref{eqn:val03} for the connectives $1$, $\neg$ and $\iimplies$, by  using rule \eqref{eqn:val} for the connectives $[C]$, where $C \in \power N$, and by putting
\begin{equation}
\Val(u, [\Oo] \phi)=\min\{\Val(v,\phi) \mid (u,v)\in R\}
\end{equation}
for any $\phi \in \Form_{\lang^+}$ and  $u\in S$.

It turns out that the class of  standard truly playable enriched $\lucas_n$-frames can be defined inside the class of standard playable enriched $\lucas_n$-frames by an $\lang^+$-formula. The next assertion is the $\lucas_n$-valued generalization of \cite[Proposition 14]{Goranko2013}.
\begin{proposition}
A standard playable enriched $\lucas_n$-frame $\framme{F}$ is truly playable if and only if $[\varnothing] \phi \leftrightarrow [\Oo]\phi$ is valid in $\framme{F}$.
\end{proposition}
\begin{proof}
First assume that $\framme{F}$ is truly playable and standard. It follows from Lemma~\ref{lem:uip} that $\br{\framme{F}}:=(S, \br{E}, R)$ is a standard truly playable enriched $\lucas_1$-frame. By \cite[Proposition 14]{Goranko2013}, for every $\phi \in \Form_{\lang^+}$ the formula $[\varnothing] \phi \leftrightarrow [\Oo]\phi$ is valid in $\br{\framme{F}}$ and
we must prove that  $[\varnothing] \phi \leftrightarrow [\Oo]\phi$ is also valid in $\framme{F}$. 

Assume that there is $\phi \in \Form_{\lang^+}$ and a model $\model{M}=(S,E,R, \Val)$ based on $\framme{F}$ such that $\Val(u, [\varnothing]\phi)<\frac{i}{n}\leq \Val(u, [\Oo]\phi)$ for some $i\leq n$. It follows from homogeneity of $E(u)$ that  $\Val(u, [\varnothing]\tau_{\sfrac{i}{n}}(\phi))=0$, while $\Val(u, [\Oo]\tau_{\sfrac{i}{n}}(\phi))=1$. Moreover, the map $\Val(-, \tau_{\sfrac{i}{n}}(\phi))$ has range in $\lucas_1$. Hence, any map $\Val'\colon S\times \Prop\to \lucas_n$ with $\Val'(w,p)=\Val(w, \tau_{\sfrac{i}{n}}(\phi))$ for every $w\in S$ defines an $\lucas_1$-model based on $\br{\framme{F}}$ that falsifies $[\varnothing]p \leftrightarrow [\Oo]p$, which is the desired contradiction. We can derive a similar contradiction in case there is some $\phi \in \Form_{\lang^+}$ and a model $\model{M}=(S,E,R, \Val)$ based on $\framme{F}$ such that $\Val(u, [\varnothing]\phi)\geq \frac{i}{n}>\Val(u, [\varnothing]\phi)$ for some $i\in \{1, \ldots, n\}$.

Conversely, let $\framme{F}=(S,E,R)$ be a standard playable enriched $\lucas_n$-frame in which $[\varnothing] \phi \leftrightarrow [\Oo]\phi$ is valid for every $\phi \in \Form_{\lang^+}$. Then $\br{\framme{F}}=(S, \br{E}, R)$ is a  standard  playable enriched $\lucas_1$-frame such that  $[\varnothing] \phi \leftrightarrow [\Oo]\phi$ is valid for every $\phi \in \Form_{\lang^+}$. Since $\br{\framme{F}}$ is truly playable by \cite[Proposition 14]{Goranko2013}, $\framme{F}$ is truly playable as well.
\end{proof}

We have to adapt the filtration technique to fit in with the newly introduced language $\lang^+$. The next definition merges Definition \ref{defn:filt} with \cite[Definition 5.3]{Teheux2014}.

\begin{definition}\label{defn:filt02}
Let $\model{M}=(S,E, R, \Val)$ be an enriched  $\lucas_n$-model and $\Gamma$ be a set of formulas closed under subformulas and the unary connectives $\tau_{\oplus}$ and $\tau_{\odot}$. With the notation introduced in Definition \ref{defn:filt}, an enriched $\lucas_n$-model $\model{M}^*=(|S|, E^*, R^*, \Val^*)$ is a $\Gamma$-\emph{filtration} of $\model{M}$ if it satisfies (\ref{it:filt01}) and (\ref{it:filt02}) of Definition \ref{defn:filt} and the following conditions:
\begin{enumerate}\setcounter{enumi}{2}
\item\label{it:bxr01} if $(u,v)\in R$, then $(|u|,|v|) \in R^*$,
\item\label{it:bxr02} if $(|u|, |v|)\in R^*$ and $u([\Oo]\phi)=1$ for every $\phi \in \Gamma$, then $v(\phi)=1$.
\end{enumerate}
\end{definition}

\begin{lemma}\label{lem:filt02}
Let $\model{M}=(S,E, R, \Val)$ be an enriched  $\lucas_n$-model and $\Gamma$ be a set of formulas closed under subformulas and the connectives $\tau_{\oplus}$ and $\tau_{\odot}$. If \[\model{M}\; \models\; \left\{[\Oo](\phi \oplus \phi) \leftrightarrow [\Oo] \phi \oplus [\Oo] \phi,\; [\Oo](\phi \odot \phi) \leftrightarrow [\Oo] \phi \odot [\Oo] \phi \right\}\] and $\model{M}^*=(|S|, E^*, R^*, \Val^*)$ is a $\Gamma$-filtration of $\model{M}$, then 
\begin{equation}\label{eqn:filt02}
\Val(u, \phi)=\Val^*(|u|, \phi)
\end{equation}
 for every $\phi \in \Gamma$.
\end{lemma}
\begin{proof}
The proof is a routine induction argument on the length of  $\phi\in \Gamma$. Considering the proof of Lemma \ref{lem:filt}, the only case we have to discuss is $\phi=[\Oo]\psi \in \Gamma$.  First, we note that by our assumption on $\model{M}$ and condition (\ref{it:bxr02}) of Definition \ref{defn:filt02},   we have
\begin{equation}\label{eqn:ert01}
\Val(u, [\Oo]\psi)\leq \Val(v, \psi)
\end{equation}
for every $(|u|, |v|)\in R^*$.

Let $u\in S$. Then the definition of $\Val^*$ and the induction hypothesis yield 
\begin{equation}\label{eqn:ert02}
\Val^*(|u|, [\Oo]\psi)=\min \{\Val(v,\psi) \mid (|u|,|v|)\in R^*\}. 
\end{equation}
The inequality $\Val^*(|u|, [\Oo]\psi)\leq \Val(u,[\Oo]\psi)$ holds true since $(|u|,|v|)\in R^*$ for every $(u,v)\in R$. The other inequality is obtained by \eqref{eqn:ert01} and \eqref{eqn:ert02}. 
\end{proof}

\begin{definition}
An \emph{$\lucas_n$-valued truly playable logic} is a subset $\logic{L}$ of $\Form_{\lang^+}$ that is closed under Modus Ponens, Uniform Substitution and Monotonicity for every $[C]$, where $C\in \power N$, and such that $\logic{L}$ contains an axiomatic base of \L ukasiewicz $(n+1)$-valued logic together with axioms (\ref{eqn:ax01})--(\ref{eqn:ax05}) of logic $\logic{P}_{n}$ and the following axioms:
\begin{framed}
\noindent \textbf{The additional axioms of $\lucas_n$-valued truly playble logic}
\begin{enumerate}\setcounter{enumi}{5}
\begin{multicols}{2}
\item\label{eqn:ax08} $[\Oo]1$
\item\label{eqn:ax06} $[\Oo] p \leftrightarrow  [\varnothing] p$,
\item\label{eqn:ax07}  $[\varnothing] (p\iimplies q)\iimplies ([\varnothing] p \iimplies [\varnothing] q)$.
\end{multicols}
\end{enumerate}
\end{framed}

\noindent
We denote by $\logic{TP}_n$ the \emph{smallest  $\lucas_n$-valued truly playable logic}.
\end{definition}
Contrary to the Boolean case, it is not known if axiom (\ref{eqn:ax07})  can be removed from the axiomatization of $\logic{TP}_n$ without changing $\logic{TP}_n$. Nevertheless, the following result holds true.

\begin{lemma} 
$\logic{TP}_n$ is closed under the necessitation rule  for $[\varnothing]$.
\end{lemma}
\begin{proof}
Let $\vdash_{\logic{TP}_n} \phi$. Then $\vdash_{\logic{TP}_n} 1 \iimplies \phi$ and $\vdash_{\logic{TP}_n} [\varnothing] 1 \iimplies [\varnothing] \phi$ by Monotonicity, which is equivalent to $\vdash_{\logic{TP}_n} [\Oo] 1 \iimplies [\Oo] \phi$ by axiom (\ref{eqn:ax06}). The conclusion follows from axiom (\ref{eqn:ax08}) and Modus Ponens.
\end{proof}

We prove completeness of $\logic{TP}_n$ with respect to the standard truly playable enriched $\lucas_n$-models by the technique of the canonical model. We denote by $\free_{\logic{TP}_n}$ the Lindenbaum-Tarski algebra of $\logic{TP}_n$.

\begin{definition}
The \emph{canonical model} of $\logic{TP}_n$ is the enriched $\lucas_n$-model $\model{M}=(\framme{F}, \Val^c)$ with $\framme{F}=(W^c,  E^c, R^c)$, where $W^c=\mathcal{MV}(\free_{\logic{TP}_n}, \lucas_n)$, $E^c$ and $\Val^c$ are as in Definition \ref{defn:canon}, and $R^c$ is defined by \[R^c=\{(u,v) \mid \forall \phi \ u([\Oo]\phi)=1 \implies v(\phi)=1\}.\]
\end{definition}

\begin{proposition}[Truth Lemma]\label{prop:tru02}
The canonical model $\model{M}=(\framme{F}, \Val^c)$ of $\logic{TP}_n$ satisfies  $\Val^c(u, \mu)=u(\mu)$ for every $\mu \in \Form_{\lang^+}$ and every $u \in W^c$.
\end{proposition}
\begin{proof}
The proof is carried out by induction on the number of connectives in $\mu$. The only case not considered in the proof of Proposition \ref{prop:tru} is $\mu=[\Oo]\phi$. However, this case was considered in the proof of \cite[Proposition 5.6]{Teheux2012} or in the proof of \cite[Proposition 5.5]{Hansoul2006}. 
\end{proof}

\noindent
It is worth noticing that Proposition \ref{prop:tru02} relies on the fact that the modality $[\Oo]$ is normal.

\begin{proposition}\label{prop:filto}
Let $\mu \in \Form_{\lang^+}$ and let $\model{M}^*=(|W^{c}|, E^{*}, \Val^{*})$ be an intermediate $\Cl(\mu)$-filtration of $(W^c,E^c, \Val^c)$. If \[R^*=\{(|u|,|v|) \mid \forall \phi \in \Cl(\mu)\ u([\Oo]\phi)=1 \implies v(\phi)=1\},\] then the model $\model{M}^+=(|W^c|, E^+, R^*, \Val^*)$ is a standard truly playable $\Cl(\mu)$-filtration of the canonical model $\model{M}=(W^c, E^c, R^c, \Val^c)$ of $\logic{TP}_n$,  where $E^+$ is obtained from $E^*$ as in Proposition \ref{prop:filt}.
\end{proposition}
\begin{proof}
By Proposition \ref{prop:filt} we know that $\model{M}$ is playable. As $\model{M}$ is finite, it is truly playable. By Proposition \ref{prop:filt} and by definition of $R^*$, the model $\model{M}^+$ is a~$\Cl(\mu)$-filtration of $\model{M}$ in the sense of Definition  \ref{defn:filt02}.

It remains to prove that $\model{M}^+$ is standard. First, assume that $u,v \in W^c$ and $E^+(|u|)(\varnothing, \neg \chi_{\{|v|\}})=0$. It follows from the definition of $E^+$ and $E^*$ that \[E(u)(\varnothing, \Val(-,\phi))=0\] for every $\phi \in \Cl(\mu)$ such that $\Val(v,\phi)=0$. We conclude that $(|u|, |v|)\in R^*$ by the definition of $R^*$.

Let $E(|u|, |v|)\in R^*$. We will prove that $E^+(|u|)(\varnothing, \neg \chi_{\{|v|\}})=0$. By way of contradiction, assume that there is $\phi \in \Cl(\mu)$ such that $\Val(v,\phi)=0$ and \[E(u)(\varnothing, \Val(-,\phi))=\frac{i}{n}>0.\] Since $\chi_{\{|v|\}}$ is idempotent and $E(u)$ is homogeneous, we may assume $i=n$. It follows from Definition \ref{defn:lplus} and  Proposition  \ref{prop:tru02} that $u([\varnothing]\phi)=1$. By axiom (\ref{eqn:ax06}) of $\logic{TP}_n$, we deduce $u([\Oo]\phi)=1$. The last identity is a contradiction since $(|u|,|v|)\in R^*$ and $v(\phi)=0$.
\end{proof}
\begin{theorem}[Completeness of $\logic{TP}_n$]\label{thm:comp02}  For any $\phi \in \Form_{\lang^+}$, the following assertions are equivalent:
\begin{enumerate}
\item\label{it:pqq01} $\vdash_{\logic{TP}_n} \phi$.
\item\label{it:pqq02} $\phi$ is true in every standard truly playable enriched  $\lucas_n$-model.
\item\label{it:pqq03} $\phi$ is true in every finite standard playable enriched $\lucas_n$-model.
\end{enumerate}
\end{theorem}
\begin{proof}
It is clear that (\ref{it:pqq02}) $\implies$ (\ref{it:pqq03}). Moreover, (\ref{it:pqq01}) $\implies$ (\ref{it:pqq02}) can be proved by a~straightforward induction argument. To prove (\ref{it:pqq02}) $\implies$ (\ref{it:pqq01}), we obtain by Proposition \ref{prop:filto} and Proposition \ref{prop:tru02} that $\phi$ is true in the canonical model of~$\logic{TP}_n$, which means $\vdash_{\logic{TP}_n} \phi$.
\end{proof}

\section{Conclusions and future research}\label{sec:further}

In this paper we have studied some generalizations of Pauly's Coalition Logic in modal extensions of \L ukasiewicz logic. Below we list some ideas for possible applications and topics for further investigations.
\begin{enumerate}
\item The gain of expressive power owing to the many-valued modal language that is used could be exploited to encode some properties of strategic or voting games, such as, for instance, the distribution of power among coalitions in~weighted voting games.  

\item In modal extensions of $(n+1)$-valued \L ukasiewicz logics, two types of relational structures can naturally be considered, giving rise to two types of~completeness results \cite{Teheux2012}. On the one hand, there is the class of frames (structures with  binary accessibility relations), while, on the other hand, there is the class of $\lucas_n$-frames. The latter are frames in which the set of~allowed truth values in a world is a prescribed subalgebra of $\lucas_n$ for every world of the frame. Such a prescription could also be considered  in the context of $\lucas_n$-valued (truly) playable logics, where the neighborhood semantics replace the relational ones. The possible aim is to obtain new completeness results with respect to this enriched semantics.

\item We have based our generalizations of Coalition Logic on modal extensions of~\L ukasiewicz  logic. Other families of many-valued logics could be considered as a basis for many-valued versions of  Coalition Logic. For example, it would be interesting to compare expressive power between the language developed in this paper and a many-valued coalitional language based on~modal extensions of G\"odel logics \cite{Metcalfe2009}. 
\item Coalition Logic is among many formal calculi developed to model the deductive aspects of games. Other systems have been considered, such as ATL \cite{Alur2002, Alur1997} and its epistemic extensions \cite{vanderhoek2003}. A natural task could be to design the many-valued versions of those calculi in order to capture wider classes of games or protocols in which errors are allowed; see \cite{Teheux2014}, for instance.

\item We did not consider the complexity issue of the satisfiability problem for the many-valued modal languages and models introduced in this paper. This topic becomes a subject of further investigation although we conjecture that the problem is \textsc{PSPACE}-hard as in the Boolean case \cite{Pauly2002}, since the number of possible truth values remains finite.

\end{enumerate}

\appendix 
\section{Representation of Boolean effectivity functions }\label{appendix:TP}
Let $N=\{1,\dots,k\}$ be a finite set of players with $k\geq 2$ and $S$ be a (possibly infinite) set of outcomes such that $|S|\geq 2$. A family $\mathbf{B}\subseteq \power S$ that contains $S$ is called a \emph{structure} on $S$. We say that $\mathbf{B}$ is \emph{closed under finite intersections} if $X_1,\ldots, X_j \in \mathbf{B}$ implies $\bigcap_{i=1}^j X_i \in \mathbf{B}$ for every $j\in \mathbb{N}$. An effectivity function $E\colon\power N \to \power\power S$ is said to be \emph{compatible} with a structure $\mathbf{B}$ on $S$ if $E(C)\subseteq \mathbf{B}$ for every $C \in \power N$, $E$ has  liveness and safety, $E(\varnothing)=\{S\}$, and $E(N)=\mathbf{B}\setminus \{\varnothing\}$.  An effectivity function is \emph{outcome monotonic with respect to $\mathbf{B}$} when the following implication holds true for every $C\in \power N$: if $X\in E(C)$, $X\subseteq Y$ and $Y\in \mathbf{B}$, then $Y\in E(C)$.

\begin{theorem}[{\cite[Theorem 3.5*]{Peleg98}}]\label{thm:pel}
Let $\mathbf{B}$ be a structure on $S$ closed under finite intersections and $E\colon \power N \to \power\power S$ be an effectivity function compatible with $\mathbf{B}$. Then the following conditions are equivalent:
\begin{enumerate}
\item $E$ is superadditive and outcome monotonic w.r.t. $\mathbf{B}$.
\item There exists a game form $G=(N, \{\Sigma_i \mid i \in N\}, S, o)$ satisfying $E(C)=H_G(C) \cap \mathbf{B}$ for every $C\in \power N$.
\end{enumerate}
\end{theorem}

As announced in Section \ref{sect:eff}, we will prove that the characterization of effectivity functions generated by game forms from \cite[Theorem 1]{Goranko2013} can be obtained as a~consequence of  Peleg's Theorem~3.5* in \cite{Peleg98}. We restate the theorem for reader's convenience. We use the notion of true playability introduced in Definition \ref{def:BooleanPlay}.

\begin{theorem}[{\cite[Theorem 1]{Goranko2013}}]
	Let $E\colon \power N \to \power\power S$ be an effectivity function. There exists a game form $G=(N, \{\Sigma_i \mid i \in N\}, S, o)$ satisfying $E=H_{G}$ if and only if $E$ is truly playable.
\end{theorem}
 \begin{proof}
 As for the first implication, it is easy to see that $H_{G}$ is playable. Set $Z=\{z\in S \mid z=o(\sigma_{N}) \text{ for some strategy profile $\sigma_{N}$}\}$. Then $Z\in H_{G}(\emptyset)$. Clearly, for any set of outcomes $X\subseteq S$ we have  $X\in H_{G}(\emptyset)$ if and only if $X$ contains the ``range'' $Z$ . This means that $H_{G}(\emptyset)$ is the principal filter generated by $Z$ and $H_{G}$ is truly playable. 
 
In order to show the converse implication, let $E$ be truly playable and put $Z=\bigcap E(\varnothing)$. Since $E$ has safety, $Z\neq \varnothing$, and since $E(\varnothing)$ is a principal filter, $E(\varnothing)=\{X  \in \power S \mid Z \subseteq X\}$. Consider the mapping $E'\colon\power N \to \power \power Z$ defined as follows:
\[
\text{$X \in E'(C)$ if $X\in E(C)$, for every $X\in \power Z$ and every $C\in \power N$.}
\]

We claim that $E'$ is compatible with the structure $\power Z$ on $Z$. Indeed, it follows that $E'(\varnothing)=\{Z\}$ and $E'(C)\subseteq \power Z\setminus\{\varnothing\}$ for every $C\in \power N$. By monotonicity, $Z\in E'(C)$ for every $C \in \power N$. It remains to prove that $\power Z\setminus \{\varnothing\}\subseteq E(N)$. Let $\varnothing \neq X \subseteq Z$. If $X=Z$, then we already know that $X\in E(N)$. Otherwise, $X$ is a~nonempty proper subset of $Z$ and thus $X\notin E(\varnothing)$. We obtain by $N$-maximality of $E$ that $\overline{X}\in E(N)$ and by superadditivity that $\overline{X}\cap Z\in E(N)$. We have proved that the complement in $Z$ of any nonempty proper subset of $Z$ is in $E(N)$, which yields $\power Z\setminus \{\varnothing\}\subseteq E(N)$. We can conclude that $E'(N)=\power Z \setminus \{\emptyset\}$.

By Theorem \ref{thm:pel}, there is a game form $G'=(N, \{\Sigma_i \mid i \in N\}, Z, o)$ such that $E'=H_{G'}$. Put $G=(N, \{\Sigma_i \mid i \in N\}, S, o)$. We will show that 
$E=H_{G}$. To this end, let $C\in \power N$ and $X \in E(C)$. By superadditivity, $X \cap Z \in E(C)$, therefore $X\cap Z \in E'(C)=H_ {G'}(C)$. By the definition of $H_{G'}$ and $H_G$, we get $X\in H_{G}(C)$. For the converse inclusion $H_{G}\subseteq E$, assume that $X\in \power S$ belongs to $H_{G}(C)$. By the definition of $H_G$, there exists $\sigma_C$ such that $o(\sigma_C\sigma_{\overline{C}})\in X\cap Z$ for every $\sigma_{\overline{C}}$. This means $X\cap Z\in H_{G'}(C)=E'(C)$, which gives $X\cap Z \in E(C)$. Finally, $X\in E(C)$ follows from outcome monotonicity.
\end{proof}

Since every filter on a finite set is principal, the class of truly playable functions and the class of playable functions coincide whenever the set of outcome states $S$ is finite. 
\section{Finite MV-algebras}\label{appendix:MV}

For a general background on {\L}ukasiewicz logic and MV-algebras see \cite{Cignoli2000,Mundici11}. In this appendix we recall the basic notions and facts about MV-algebras that are needed in this paper.

An \emph{MV-algebra} is an algebra $(A,\oplus,\neg,0)$, where $\oplus$ is a~binary operation, $\neg$ is a~unary operation and $0$ is a constant, such that the following equations are satisfied:
\begin{enumerate}
	\item $(A,\oplus,0)$ is an Abelian monoid,
	\item $\neg(\neg x)=x$,
	\item $\neg 0 \oplus x=\neg 0$,
	\item $\neg (\neg x \oplus y) \oplus y=\neg(\neg y\oplus x) \oplus x$.
\end{enumerate}
\noindent 
We introduce the new constant $1$ and two additional operations $\odot$ and $\to$ as follows:
\begin{align*}
1 &= \neg 0,\\  x\odot y&=\neg (\neg x \oplus \neg y),\\  x\to y&=\neg x \oplus y.
\end{align*}
 
 We say that an MV-algebra  $(A,\oplus,\neg,0)$ is \emph{finite} whenever $A$ is finite.
As usual we will say that ``$A$ is an MV-algebra'' when no danger of confusion arises. For every MV-algebra $A$, the binary relation $\leq$ on $A$ given by
\[
x\leq y \quad\text{ whenever}\quad  x \rightarrow y=1
\]
is a partial order. As a matter of fact, $\leq$ is a lattice order induced by the join $\vee$ and the meet $\wedge$ operations defined by
\begin{align*}
x\vee y&=\neg (\neg x \oplus y) \oplus y, \\ x\wedge y&=\neg(\neg x \vee \neg y),
\end{align*}
respectively. Thus defined, the lattice reduct of $A$ is a distributive lattice with  top element $1$ and  bottom element $0$. If the order $\leq$ of $A$ is total, then $A$ is said to be an \emph{MV-chain}.

The algebraic semantics of finite-valued {\L}ukasiewicz logics is given by finite MV-chains. The standard example of a finite MV-chain is a \emph{finite {\L}ukasiewicz chain} given by
 \[
\lucas_n=\left\{0, \tfrac{1}{n}, \ldots, \tfrac{n-1}{n}, 1\right\}, \quad \text{where $n$ is a positive integer.}
\]
For every $x,y\in\lucas_n$, put 
\begin{align*}
\neg x &=1-x, \\
x \oplus y& =\min(x+y, 1)
\end{align*}
Then $(\lucas_n,\oplus,\neg,0)$ becomes an MV-chain, where the lattice operations $\wedge$ and $\vee$ are the minimum and the maximum of $x,y\in\lucas_n$,  respectively. Further derived operations $\odot$, $\rightarrow$ and $\leftrightarrow$ on $\lucas_n$ are given by
\begin{align*}
x\odot y&=\max (x+y-1, 0), \\
x\iimplies y&=\min(1, 1-x+y),\\
x\leftrightarrow y&=1-\lvert x-y\rvert.
\end{align*}
 Observe that the choice $n=1$ gives a two-element {\L}ukasiewicz chain $\lucas_1=\{0,1\}$, in which $\oplus$ coincides with $\vee$ and $\odot$ coincides with $\wedge$. The semantics of classical propositional logic is thus determined by $\lucas_1$. On the other hand, the algebraic semantics of  finite $(n+1)$-valued {\L}ukasiewicz logic with $n\geq 2$ is given by the variety of MV-algebras $\var{MV}_n$ that is axiomatized by Grigolia's axioms \cite{Grigolia77}:
\begin{enumerate}
	\item $\bigodot_{i=1}^n x= \bigodot_{i=1}^{n+1} x$,
	\item $\bigoplus_{i=1}^{n+1} \bigodot_{j=1}^m x=\bigodot_{i=1}^{n+1}\bigoplus_{j=1}^m\bigodot_{k=1}^{m-1} x$,
\end{enumerate}
for every integer $m\in \{2,\dots,n-1\}$ that does not divide $n$. Moreover, it is known that $\var{MV}_n$ is generated (as a variety) by the {\L}ukasiewicz chain $\lucas_n$.

An \emph{MV-filter} (or a \emph{filter}) in an MV-algebra $A$ is a subset $F\subseteq A$ such that
\begin{enumerate}
	\item $1 \in F$,
	\item if $x,y\in F$, then $x\odot y\in F$,
	\item if $x\in F$ and $x\leq y \in A$, then $y\in F$.
\end{enumerate}
\noindent
A \emph{principal filter} in $A$ is a filter $F$ for which there exists $x\in A$ such that $F$ coincides with the smallest filter containing $x$. If $A$ is finite, this means simply $F=\{y\in A \mid y\geq \bigodot_{i=1}^n x \}$ for some $x\in A$.

\end{document}